\documentclass[10pt]{article}

\usepackage{amsfonts,epsfig}

\usepackage{amsfonts}
\usepackage{amsmath,amssymb,amsthm}
\usepackage{url}

\newtheorem{Th}{Theorem}

\newtheorem{Lemma}{Lemma}
\newtheorem{Cor}{Corollary}

\newtheorem*{Prob}{Problem}

\theoremstyle{remark}
\newtheorem{Ex}{Example}
\newtheorem{Rem}{Remark}

\newcommand{\veps}{\varepsilon}
\newcommand{\p}{\partial}
\newcommand{\extd}{\mathrm{d}}

\newcommand{\w}{\wedge}
\newcommand{\bbR}{\mathbb{R}}
\newcommand{\eup}{\mathfrak{p}}

\newcommand{\eus}{\mathfrak{s}}
\newcommand{\eusl}{\mathfrak{sl}}
\newcommand{\eugl}{\mathfrak{gl}}
\newcommand{\ev}{\operatorname{ev}}

\newcommand{\spn}{\mbox{\rm span}}
\newcommand{\weg}[1]{}

\begin{document}

\title{A solution of a problem of Sophus Lie: Normal forms of 2-dim metrics admitting two projective vector fields.}

\author{Robert L. Bryant%
\thanks{Duke University Mathematics, PO Box 90320, Durham, NC 27708}, 
Gianni Manno%
\thanks{Department of Mathematics,  via per Arnesano, 73100 Lecce Italy}, 
Vladimir S. Matveev%
\thanks{Institute of Mathematics, FSU Jena, 07737 Jena Germany,  matveev@minet.uni-jena.de}}

\date{}

\maketitle

\begin{abstract}
We give a complete list of normal forms for the $2$-dimensional metrics 
that admit a transitive Lie pseudogroup of geodesic-preserving 
transformations and we show that these normal forms are mutually 
non-isometric. This solves a problem posed by Sophus Lie.
\end{abstract}

\section{Introduction}

\subsection{Main definition and result}

Let $g$ be a Riemannian or pseudo-Riemannian metric 
on a 2-dimensional surface $D$. A diffeomorphism~$\phi:U_1\to U_2$
between two open subsets of~$D$ is said to be \emph{projective} 
with respect to~$g$ if it takes the unparametrized geodesics of~$g$ 
in~$U_1$ to the unparametrized geodesics of~$g$ in~$U_2$. 

Lie showed that what is (nowadays) called the pseudo-group~$P(g)$ 
of projective transformations of~$g$ is a Lie pseudo-group. 
A vector field $v$ on~$D$ is said to be \emph{projective} with respect to~$g$, 
if its (locally defined) flow belongs to~$P(g)$. As Lie showed,
the set of vector fields projective with respect to a given $g$ 
forms a Lie algebra. We will denote this Lie algebra by $\eup(g)$.

Clearly, any Killing vector field for $g$ is projective with respect to $g$. Metrics of constant curvature give examples of metrics admitting projective vector fields that are not Killing. It has been known 
since the time of Lagrange and Beltrami \cite{Beltrami,lagrange} 
that a metric~$g$ of constant curvature on a simply connected
domain~$D$ has $\dim\left(\eup(g)\right)=8$. (In fact, in this case,
$\eup(g)\simeq\eusl(3,\bbR)$ as Lie algebras.)

The following problem was posed by Sophus Lie%
\footnote{
German original from \cite{Lie}, Abschn. I, Nr. 4, Problem II: 
\emph{Man soll die Form des Bogenelementes einer jeden Fl\"ache bestimmen, 
deren geod\"atische Kurven mehrere infinitesimale Transformationen gestatten}.} 
in 1882:

\begin{Prob}[Lie] 
Find all metrics $g$ such that $\dim\eup(g)\ge 2$.
\end{Prob}

The following theorem, which is the main result of our paper, 
essentially solves this problem.

\begin{Th} \label{main}
Suppose that a metric $g$ on~$D^2$ possesses two projective vector fields 
that are linearly independent at some~$p\in D^2$. Then 
either~$g$ has constant curvature on some neighborhood of~$p$ 
or else, on some neighborhood of~$p$, there exist coordinates~$x,y$ 
in which the metric~$g$ takes one of the following forms:

\begin{enumerate} 
\item Metrics with $\dim\eup(g)=2$.

\begin{enumerate}

\item $\veps_1 e^{(b+2)\,x}\,dx^2 
    +\veps_2 e^{b\,x}\,dy^2 $, \label{1a} 
where $b\in \mathbb{R}\setminus \{-2,0,1\}$ 
and $\veps_i\in \{-1,1\}$ are constants,

\item \label{1b} 
 $a\left( \frac{ e^{(b+2) \, x} dx^2}
                           {(e^{b\,x } +\veps_2)^2} 
     + \veps_1\frac{ e^{b\,x}dy^2}{e^{b\,x }+\veps_2}\right)$, 
where $a\in\mathbb{R}\setminus\{0\}$, $b\in \mathbb{R}\setminus \{-2,0,1\}$,  
and $\veps_i\in \{-1,1\}$ are constants, and

\item \label{1c}
$a\left(\frac{e^{2\,x}dx^2}{x^2}+\veps \frac{dy^2}{x}\right)$, 
where $a\in\mathbb{R}\setminus\{0\}$, and $\veps\in \{1, -1\}$ are constants.

\end{enumerate}

\item Metrics with $\dim\eup(g)=3$.

\begin{enumerate}

\item \label{corol2}
$\veps_1 e^{3x}dx^2+\veps_2 e^{x} dy^2$, 
where $\veps_i\in \{-1,1\}$ are constants, 

\item \label{corol1}
$a\left(\frac{e^{3x}dx^2}{(e^{x}+\veps_{2})^{2}} +
\veps_1\frac{e^{x}dy^2}{(e^{x}+\veps_{2})}\right)$,
where $a\in \mathbb{R}\setminus\{0\}$, $\veps_i\in \{-1,1\}$ are constants, and

\item \label{case2}
$a\left(\frac{dx^2}{(cx+2x^2+\veps_2)^2 x} 
          + \veps_1\frac{x dy^2}{ (cx+2x^2+\veps_2)}\right)$,
where $a>0$, $\veps_i\in \{-1,1\}$, $c\in \bbR$ are constants. 

\end{enumerate}

\end{enumerate}

No two distinct metrics from this list are isometric.

\end{Th}

Strictly speaking, what Theorem~\ref{main} 
gives is a list of local normal forms for metrics~$g$ 
whose projective pseudo-group~$P(g)$ is locally transitive.
Naturally, for such metrics, one has~$\dim\eup(g)\ge2$.
In Sections~\ref{phenomena} and~\ref{phenomena1} it will be shown that, 
for a metric~$g$ of nonconstant curvature on a surface~$D$, 
any two-dimensional subalgebra~$\eus\subset\eup(g)$ 
acts locally transitively on a dense open subset of~$D$.  
(One can further show that this open set has full measure.)
Thus, if~$\dim\eup(g)\ge2$, then either~$g$ 
has constant curvature or else, outside a closed set with no interior,%
\footnote{Examples show that this closed set can be nonempty.}
$g$ is locally isometric to one of the metrics given in Theorem~\ref{main}.

\begin{Rem}
The higher dimensional analog of Lie's question is easier, 
possibly because the systems of PDE that appear have a
higher degree of overdeterminacy. It was treated with success by
Solodovnikov~\cite{solodovnikov1,solodovnikov2, solodovnikov25}.
The global (i.e., when the manifold and the projective vector fields  are  complete) version
of Lie's question was explicitly asked by Schouten
\cite{schouten}; for the Riemannian case, 
it was solved in \cite{Obata,lich,CMH,archive}.
\end{Rem}

\subsection{Killing vector fields for metrics 
with $\dim\eup(g)\ge 2$ and results of A.V. Aminova}

It appears to be commonly believed that A.V. Aminova solved
Lie's problem in \cite{Aminova0}, \cite{Aminova2}. Indeed, she claimed

\begin{Th} \label{killing} 
If a metric~$g$ satisfies~$\dim\eup(g)\ge2$, 
then it admits a Killing vector field.
\end{Th}

\noindent She then used this result to classify the metrics~$g$
with $\dim\left(\eup(g)\right)\ge 2$.

Unfortunately, Aminova's proof of Theorem~\ref{killing}
has a serious gap. Namely, she assumed 
(see \cite[pp.~3,6]{Aminova0} or \cite[pp.~414,424]{Aminova2}) 
that Koenigs~\cite{konigs} had proved that if the geodesic flow
of a metric admits \emph{three} linearly independent quadratic integrals%
\footnote{For a discussion of the notion of `quadratic integrals' 
of a geodesic flow and their independence, 
see Section~\ref{superintegrable}.}%
, then the metric admits a nontrivial Killing vector field.

However, a careful reading of Koenigs' note~\cite{konigs} 
shows that he had instead proved that
if the geodesic flow of a metric admits \emph{four} linearly
independent quadratic integrals, then the metric
admits a nontrivial Killing vector field.

Indeed, Koenigs' examples already show that the existence 
of three linearly independent quadratic integrals 
does not imply the existence of a nontrivial Killing vector field.

\begin{Ex}[{\upshape\cite{konigs}}] Consider the metric $g$ given by
$(4x^2{+}y^2{+}1)(dx^2+dy^2)$. The geodesic flow of this metric has
three linearly independent quadratic integrals. They are
\begin{equation*}
\begin{aligned}
F_0&=(4x^2{+}y^2{+}1)(dx^2+dy^2),\\ 
F_1&=(4x^2{+}y^2{+}1)\bigl(y^2\,dx^2-(4x^2{+}1)\,dy^2\bigr),\\
F_2&=(4x^2{+}y^2{+}1)\bigl(xy^2(dx^2+dy^2)
       +(4x^2{+}y^2{+}1)dy\,(x\,dy-y\,dx)\bigr).
\end{aligned}  
\end{equation*}
However this metric admits no nontrivial Killing vector field.
Indeed, the scalar curvature $R$ of this metric 
and its $g$-Laplacian $\Delta_g R$ are independent functions,
which is impossible for metrics admitting nontrivial Killing vector fields. 
\end{Ex}

On the other hand, 
all the normal forms listed in Theorem~\ref{main} admit $\frac{\p }{\p y}$ 
as a Killing vector field. Thus, Theorem~\ref{main} implies 
Theorem~\ref{killing} (in  the proof of Corollary~\ref{corlst} we will show, that if a projective vector field is a Killing vector field 
 on an open subset, then it is  a Killing vector field on the whole connected manifold).  Consequently, 
Aminova's description of the metrics~$g$ with $\dim\eup(g)\ge 2$ 
is correct and can be obtained from ours by coordinate changes. 
However, we prefer our description because it is simpler.
For example, all the metrics in our list are given by elementary functions, 
while some of Aminova's metrics include functions implicitly given 
as a solution of a certain differential equation.

Finally, we would like to remark that we have not found a way to
prove Theorem~\ref{killing} directly, i.e., without constructing
the list of the metrics first.

\subsubsection*{Acknowledgements}
The first author acknowledges financial support from the
National Science Foundation via DMS-0604195 and from Duke University. 
The second author would like  
to express his gratitude to Katholieke
Universiteit Leuven for hospitality,
 and to R. Vitolo for useful discussions.
The third author would like to thank  V.~Bangert, A.~Bolsinov, A.~Cap,
B. Doubrov,  M.~Eastwood,  O.~Kowalski,  B.~Kruglikov,    and V.~Shevchishin for useful discussions, 
Deutsche Forschungsgemeinschaft 
(Priority Program 1154 --- Global Differential Geometry), 
Ministerium f\"ur Wissenschaft, Forschung und Kunst
Baden-W\"urttemberg  (Elite\-f\"orderprogramm Postdocs 2003),  
KU Leuven,  and FSU Jena for partial financial support, 
and University of Lecce and Duke University for hospitality. 

\section{Proof}

\subsection{Outline of the proof}

In Section~\ref{0} we recall some classical results of Lie, Liouville,
Tresse, Cartan, and Koenigs. In Section~\ref{connection} we use these
results to describe all projective connections admitting two
projective vector fields. In Sections~\ref{miracle} and
\ref{metric} we determine which of these connections come from
(pseudo)Riemannian metrics. In order to do this, given a projective
connection, we first find a system of PDE such that the
existence of nontrivial solutions of this system implies the
existence of metrics with this projective connection 
(see Lemma~\ref{1} in Section~\ref{miracle}). This equation
(\eqref{lin1} in Section~\ref{miracle}) is seen to be linear%
\footnote{This linearity observation is a coordinate-free
version of an observation that Darboux attributed to Dini. 
See Darboux~\cite[\S\S600--608]{Darboux}.  
However, the general result was already known 
to R. Liouville in 1887~\cite{liouville}.}
in the components of the matrix $g/\det(g)^{2/3}$.

In Section~\ref{metric} we use the linearity of the equations 
and the good normal form for the projective connection from
Section~\ref{connection} to reduce this system of partial
differential equations to a system of ordinary differential
equations under the additional assumption that the metric is not
superintegrable%
\footnote{We recall the definition of
superintegrable metrics in Section~\ref{superintegrable}.
We also prove that superintegrable metrics are those admitting precisely
three linearly independent projective vector fields.}.
Then we solve this system of ODE and obtain the list
of the metrics. In Section~\ref{2.5} we describe all metrics~$g$
with nonconstant curvature and~$\dim\eup(g)\ge3$ 
(which turn out to be preciesly the superintegrable ones, see
 Lemma~\ref{superintegrable1}  and Corollary~\ref{corlst}). 
Finally, in Section~\ref{different} we explain why the metrics 
from Theorem~\ref{main} are mutually nonisometric.

\subsection{ Classical results of Beltrami, Lie, Liouville,  Koenigs,  Cartan,  and Tresse} 
\label{0}

\subsubsection{Projective connections} \label{01}
A second order ordinary differential equation of the form
\begin{equation}\label{eq1}
y'' = K^0(x,y) + K^1(x,y)\,y' + K^2(x,y)\,(y')^2 + K^3(x,y)\,(y')^3,
\end{equation}
where the functions~$K^i$ are defined on some connected domain~$D$ 
in the $xy$-plane, is classically referred 
to as a \emph{projective connection}.%
\footnote{For the relation of this classical notion 
with the modern formulation of projective connections due to Cartan, 
see~\cite{cartan}.}

For any symmetric affine connection on the domain $D^2$ 
with coordinates $x,y$, say,
$\Gamma(x,y)=\bigl(\Gamma^i_{jk}(x,y)\bigr)
=\bigl(\Gamma^i_{kj}(x,y)\bigr)$, 
the projective connection \emph{associated to}~$\Gamma$ 
is defined to be
\begin{equation} \label{con} 
y'' =-\Gamma^2_{11}+(\Gamma^1_{11}{-}2\Gamma^2_{12})\,y'
       -(\Gamma^2_{22}{-}2\Gamma^1_{12})(y')^2
         +\Gamma^1_{22}(y')^3.
\end{equation} 
It has been known since the time of Beltrami~\cite{Beltrami}, 
that the solutions of this ODE and the geodesics 
of the connection~$\Gamma$ are closely related: 
Namely for every solution $y(x)$ of~\eqref{con}
the curve $\bigl(x,y(x)\bigr)$ is, up to reparametrization, 
a geodesic of~$\Gamma$.

It is well-known (and easy to verify directly) that, 
under any coordinate change of the form 
$(x,y)=\bigl(F(\bar x, \bar y),G(\bar x,\bar y)\bigr)$,
the projective connection~\eqref{eq1} 
is transformed into another projective connection, 
now expressed in terms of~$(\bar x, \bar y)$. 

A vector field%
\footnote{When the $xy$-coordinates are clear from context, 
we will also use the more compact notation $Z = (Z^1, Z^2)$ 
for this vector field.}
$$
Z = Z^1(x,y)\frac{\p}{\p x}+Z^2(x,y)\frac{\p}{\p y}
$$ 
on~$D$ is an \emph{infinitesimal symmetry} 
of the projective connection~\eqref{eq1} 
if its (local) flow preserves~\eqref{eq1}.
It is known \cite{Lie, Tresse}, 
that a vector field~$Z$ as above is an infinitesimal symmetry
of~\eqref{eq1} if and only if it satisfies the PDE system
\begin{equation} \label{ptr}
\left.
\renewcommand{\arraystretch}{1.5}
\begin{array}{r}
Z^2_{xx}-2K^0Z^1_x-K^1Z^2_x+K^0Z^2_y-K^0_xZ^1-K^0_yZ^2=0\\
-Z^1_{xx}+ 2Z^2_{xy}-K^1Z^1_x-3K^0Z^1_y-2K^2Z^2_x-K^1_xZ^1-K^1_yZ^2=0\\
-2Z^1_{xy}+Z^2_{yy}-2K^1Z^1_y-3K^3Z^2_x-K^2Z^2_y-K^2_xZ^1-K^2_yZ^2=0\\
-Z^1_{yy}+K^3Z^1_x-K^2Z^1_y-2K^3Z^2_y-K^3_xZ^1-K^3_yZ^2=0
\end{array}
\right\}.
\end{equation}

The set of vector fields in~$D$
whose flows preserve the projective connection~\eqref{eq1}
form a Lie algebra, which will be denoted~$\eup(D,\eqref{eq1})$,
or, more simply,~$\eup(\eqref{eq1})$ 
when the domain~$D$ is clear from context.
We will also use the notation~$\eup(\Gamma)$ for the equation~\eqref{con},
and for a metric~$g$, we will use the notation~$\eup(g)$ to denote
infinitesimal symmetries of the the projective connection associated 
to the Levi-Civita connection of~$g$.

For use below, we want to point out the following consequence of~\eqref{ptr}:
Differentiating the equations~\eqref{ptr} with respect to~$x$ and~$y$, 
one obtains a system of $8$ third order equations for the~$Z^i$ 
that can be solved for all of the third order derivatives of the~$Z^i$ 
as linear expressions in their lower order derivatives.

In particular, consider the~$\bbR^8$-valued function
$$
\hat Z = (Z^1,Z^2,Z^1_x,Z^2_x,Z^1_y,Z^2_y,Z^1_{xy},Z^2_{xy}).
$$
The system~\eqref{ptr} and its derivatives can be written 
in the vectorial form
$$
\extd\hat Z = \hat Z\,( X\,\extd x + Y\,\extd y).
$$
where~$X=X(x,y)$ and~$Y=Y(x,y)$ are certain $8$-by-$8$ matrices 
whose entries are constructed from the functions~$K^i$ 
and their first two derivatives.  
One can thus regard the $\eugl(8,\bbR)$-valued 
$1$-form $\psi = (X\,\extd x+Y\,\extd y)$ 
as defining a linear connection on the bundle~$D\times\bbR^8$,
one whose parallel sections~$\hat Z$ on an open set~$U\subset D$
correspond to the elements of~$\eup(U,\eqref{eq1})$.
This interpretation of~\eqref{ptr} has some useful consequences.%
\footnote{For those interested in an invariant formulation, 
we offer the following description:  The linear differential
equations~\eqref{ptr} define in~$J^2(D,TD)$, the bundle of $2$-jets
of vector fields on~$D$, a vector subbundle~$P\subset J^2(D,TD)$
of rank~$8$.  The contact plane field on~$J^2(D,TD)$ (i.e., the plane field
of codimension~$6$ to which all of the holonomic sections of~$J^2(D,TD)$
are tangent) produces a horizontal $2$-plane field~$H\subset TP$
that defines a linear connection on the bundle~$P\to D$.  
The $2$-jets of projective vector fields on~$D$ 
are then sections of~$P$ that are parallel with respect to this connection.}

First, if~$U\subset D$ is connected, then,
for any point~$p\in U$, the evaluation mapping
$$
\ev_p(Z)=\bigl(Z^1(p),Z^2(p),Z^1_x(p),Z^2_x(p),Z^1_y(p),Z^2_y(p),
Z^1_{xy}(p),Z^2_{xy}(p)\bigr)
$$
defines a linear injection~$\ev_p:\eup(U,\eqref{eq1})\to\bbR^8$. 
In particular, if any~$Z\in\eup(\eqref{eq1})$
vanishes to order~$3$ at any point of~$D$ (assumed connected), 
then~$Z$ vanishes identically.

Second, for any~$p\in D$, one can define its \emph{local} infinitesimal
symmetry algebra~$\eup(p,\eqref{eq1})$ to be the inverse limit of the
symmetry algebras~$\eup(U,\eqref{eq1})$ as~$U$ ranges over the open
neighborhoods of~$p$ in~$D$.  The above formulation then implies that
$\eup(p,\eqref{eq1})=\eup(U,\eqref{eq1})$ for some connected 
open neighborhood~$U$ of~$p$ in~$D$. 

Third, differentiating~$\extd\hat Z = \hat Z\,\psi$ yields
$0 = \extd(\extd\hat Z) = \hat Z (\extd\psi + \psi\w\psi)$, 
so that one gets the relation~$\hat Z L = 0$, where
$$
\extd\psi + \psi\w\psi = (Y_x - X_y + [X,Y])\,\extd x\w\extd y 
   = L(x,y)\,\extd x\w\extd y
$$
is the curvature form of the connection~$\psi$.  Thus, if~$L(p)$
is nonzero, one gets that $\dim\eup(p,\eqref{eq1})<8$.  
Also, when~$L$ is nonzero, one gets more relations on~$\hat Z$ 
by differentiating~$0=\hat Z L$, 
which yields~$0 = \hat Z (\extd L + \psi\,L)$.  On the other hand,
when~$L$ vanishes identically on a simply-connected open 
$p$-neighborhood~$U$, the flatness of the connection~$\psi$ 
implies that~$\dim\eup(U,\eqref{eq1})=8$.

\subsubsection{Lie algebras of projective vector fields}
\label{02} 
Lie~\cite{Lie} classified the possible local infinitesimal symmetry algebras 
of projective vector fields of projective connections. 
He proved\footnote{Modern version of  Lie's proof can be found, for example, in  \cite{ibragimov,romanovski}.} that such a local Lie algebra 
is isomorphic to one of the following
$$
1.\ \{0\},\qquad
2.\ \bbR,\qquad
3.\ \eus,\qquad
4.\ \eusl(2,\bbR),\qquad
5.\ \eusl(3,\bbR).
$$
where~$\eus$ is the non-commutative Lie algebra of dimension~$2$ 
(i.e., the algebra spanned by two elements~$X$ and~$Y$ that satisfy~$[X,Y]=X$). 
Since we will be considering the cases when a metric has at least
two linearly independent projective vector fields, 
the first two algebras will not arise.

Both of the algebras $\eusl(2,\bbR)$ and $\eusl(3,\bbR)$ 
contain~$\eus$ as a subalgebra. Thus, by Lie's classification,
\emph{if a connection $\Gamma$ satisfies~$\dim\eup(\Gamma)\ge2$, 
it admits two projective vector fields~$X\not\equiv0$ and $Y$ 
satisfying the relation $[X,Y]=X$.}

Lie~\cite{lie2} also investigated the possible realizations of~$\eus$
as an algebra of vector fields on $\bbR^2$. He showed that, 
for two vector fields $X,Y$ on $\bbR^2$ satisfying the relation $[X,Y]=X$, 
almost every point $p\in\bbR^2$ has a neighborhood 
on which there are coordinates $x,y$ in which
\begin{itemize}
\item[] {\bf transitive case } 
 $X=\frac{\p }{\p y} $,
 $Y=\frac{\p }{\p x} + y \frac{\p }{\p y} $,
 \label{nt0}
\item[] {\bf nontransitive case } 
 $X=e^y \frac{\p }{\p y}$, 
 $Y=- \frac{\p }{\p y}$, or \label{nt}
\item[] {\bf trivial case } $X\equiv 0$.
\end{itemize}

\label{phenomena}

Note that there exist examples of smooth vector fields $X\not\equiv 0$ and~$Y$ on the plane 
that satisfy the relation $[X,Y]=X$ and are
such that they are as in the {\bf transitive} case in some open set 
and as in the {\bf trivial} case in another open set. 

\begin{Ex} Let the function $f:\bbR\to\bbR$ be  given by 
$$
f(y)=\begin{cases}e^{-\frac1{2y^2}} \ & \text{for $y>0$}, \\
         0 & \text{for $y\le0$.} 
   \end{cases} 
$$
Then the vector fields $X ={f(y)}\frac{\p }{\p x}$
and $Y =-y^3\frac{\p }{\p y}$,
satisfy $[X,Y]=X$ and are as in the {\bf transitive}
case for $y>0$ but are as in the {\bf trivial } case for $y<0$.

Moreover, the vector fields~$X= \frac{\p }{\p x}$
and~$Y = x\frac{\p }{\p x} + f(y)\frac{\p }{\p y}$
also satisfy~$[X,Y]=X$ and fall into the transitive case for~$y>0$
but into the nontransitive case for~$y<0$.
\end{Ex}

 From the results of Section~\ref{01} it follows, that 
this phenomena cannot happen,  if the vector fields are infinitesimal symmetries of a projective connection (on a connected $D$). Indeed,   if  $X \in \mathfrak{p} (\eqref{eq1})$   vanishes at every point of  
 an open  subset $U$, then for every  $p\in U$  
  we evidently have  $\ev_p(X)=0$ implying     $X\equiv 0$. 
  
 Thus,  \emph{ if  two infinitesimal  symmetries  $X,Y$  of a projective connection on a connected  surface  satisfy  $[X,Y]=X$ and $X\not\equiv  0$, 
 then,   in a neighborhood of almost every point,  they  
 are as in the {\em 
  transitive case}, or as in the {\em  nontransitive case}.    }

\subsubsection{ Invariants that decide whether a projective connection corresponds to the metric of constant curvature.} \label{phenomena1}

We say that a projective connection on $D^2$ is \emph{flat},
if every point of $D^2$ has local coordinates $x,y$ such that the
projective connection has the form $y''=0$ (i.e., if the geodesics
of the projective connection can be mapped into straight lines).
Lie~\cite{Lie} showed that the projective connection of a metric 
is flat if and only if the metric has constant curvature, 
and that the (local) Lie algebra $\eup(g)$ of a metric $g$ 
of constant curvature is $\eusl(3,\bbR)$. (Of course, this follows
immediately from Beltrami's earlier result that the geodesics
of a metric can be mapped to straight lines in the plane if and
only if the metric has constant curvature.)

There is a simple test for when the projective connection~\eqref{eq1} is flat. 
Consider the following two functions 
(sometimes improperly called \emph{Cartan invariants}
although they were already known to Liouville~\cite{liouville} in 1889):
\begin{equation}\label{eq: Liouvilleinvariants}
\begin{aligned}
L_1&= 2K^1_{xy}-K^2_{xx}-3K^0_{yy}- 6K^0K^3_x- 3K^3K^0_x\\
&\qquad + 3K^0K^2_y + 3K^2K^0_y + K^1K^2_x - 2K^1K^1_y\\
L_2&= 2K^2_{xy} - K^1_{yy} - 3K^3_{xx} + 6K^3K^0_y + 3K^0K^3_y\\
&\qquad - 3K^3K^1_x - 3K^1K^3_x - K^2K^1_y + 2K^2K^2_x
\end{aligned}
\end{equation}
Liouville~\cite{liouville} proved that the expression
\begin{equation*}
\lambda = (L_1\,\extd x + L_2\,\extd y)\otimes (\extd x \w \extd y)
\end{equation*}
is a differential invariant%
\footnote{In fact, it is the lowest order tensorial invariant.} (w.r.t.  coordinate changes) 
of the projective connection~\eqref{eq1}.  
Moreover, Liouville~\cite{liouville}, Tresse~\cite{Tresse}, 
and Cartan~\cite{cartan} each gave independent proofs that a projective
connection is flat on an open set~$U$ if and only~$\lambda$ vanishes there.
\label{03}

\begin{Rem} 
 A calculation shows that
the equation~$\hat Z L = 0$ from Section~\ref{01}  
is just the condition that the Lie derivative of
$\lambda$ with respect to~$Z$ be zero, see for example \cite{romanovski}. \end{Rem}

A direct corollary of this classical result is the following:

\begin{Lemma} \label{lem3} 
Consider the projective connection \eqref{eq1} on a connected domain~$D$
and assume that its Liouville invariant~$\lambda$ is nonvanishing on~$D$.
Suppose that vector fields $X$ and~$Y$ in~$\eup(\eqref{eq1})$ 
satisfy $[X,Y]=X$ and that $X$ is not identically~$0$.
Then, at every point of  a dense open subset of $D$  
the vector fields $X$ and $Y$ are linearly independent 
{\upshape(}and hence fall into the \emph{transitive} case{\upshape)}.
\end{Lemma}

\begin{proof} As we explained in the previous section, 
$D$ contains a dense open subset on which the pair~$(X,Y)$ 
falls into either the transitive case or the nontransitive case.
\weg{(Since~$D$ is connected and~$X\in\eup(\eqref{eq1})$ 
does not vanish identically, we have already seen that~$X$ is nonzero
on a dense open subset of~$D$.)}
We need only show that the invariant~$\lambda$ 
vanishes on any open set on which~$(X,Y)$ falls into the nontransitive case.

Thus, assume that we are in the nontransitive case, i.e.,
that one can choose local coordinates~$x$ and~$y$ such that
$X=e^y \frac{\p }{\p y}$ and $Y=-\frac{\p }{\p y}$. 
Since $Y$ belongs to~$\eup(\eqref{eq1})$, for every solution $y(x)$ 
of~\eqref{eq1} and for every $t\in\bbR$, we have that $y(x)+t$ 
is also a solution. It follows without difficulty 
that the functions $K^i$ depend on~$x$ only.
Substituting $Z:=X=(0, e^y)$ in the equations~\eqref{ptr}, 
we obtain $K^0=K^2=K^3=0$, which, together with $K^1=K^1(x)$,
implies, by the formulae~\eqref{eq: Liouvilleinvariants}, 
that~$L_1=L_2=0$ and hence~$\lambda=0$.  The Lemma is proved.
\end{proof}
   
\subsubsection{Superintegrable metrics  have $\dim(\mathfrak{p})=3$}
\label{superintegrable} 

Let~$g$ be a (pseudo-)Riemannian metric on a connected surface~$M$.
The geodesic flow of~$g$ is then a well-defined flow on~$TM$.
A function~$h:TM\to\bbR$ is an \emph{integral} of the geodesic
flow of~$g$ if it is constant on the orbits of the geodesic
flow of~$g$. The necessary and sufficient condition that~$h$
be an integral of the geodesic flow of~$g$ is that it satisfy
the linear first order PDE
\begin{equation*}
\{h,g\}_g = 0
\end{equation*}
where~$\{,\}_g$ is the Poisson bracket on~$TM$ that is transferred
from the canonical one on~$T^*M$ 
via the bundle isomorphism~$\flat_g:TM\to T^*M$.

An integral~$h$ of the geodesic flow of~$g$ 
is said to be a \emph{quadratic integral}
if it is a quadratic function on each tangent space~$T_pM$. 
For example~$g$ itself is a quadratic integral of the geodesic flow.%
\footnote{The importance of quadratic integrals other than $g$ itself
for studying the geodesic flow was recognized long ago. Indeed,
it was Jacobi's realization that the geodesic flow of the ellipsoid
admitted such an `extra' quadratic integral that allowed him to 
integrate the geodesic flow on the ellipsoid.}

For quadratic forms~$h:TM\to\bbR$, 
the linear equation $\{h,g\}_g = 0$ 
is an overdetermined PDE system of finite type.
Hence, the quadratic integrals of the geodesic flow of~$g$
always form a finite dimensional vector space~$\mathcal{I}(g)$.
Koenigs~\cite{konigs} proved that, on a connected
surface~$M$, the dimension of $\mathcal{I}(g)$ 
is $1$, $2$, $3$, $4$, or~$6$.
Moreover, if~$\dim\left(\mathcal{I}(g)\right)=6$, then~$(M,g)$ is isometric
to a connected domain in a  surface of constant curvature. 
The metrics~$g$ of nonconstant curvature such that $\dim\left(\mathcal{I}(g)\right)=4$ 
(the next highest value) are sometimes said to be \emph{superintegrable} 
(or \emph{Darboux-suprintegrable}) \cite{miller}.

Two metrics $g$ and $\bar g$ on (the same) $D^2$ 
are said to be \emph{projectively equivalent} if they have
the same geodesics, considered as unparameterized curves. 
The following result connecting projective equivalence and
the space~$\mathcal{I}(g)$ is proved in Darboux~\cite[\S608]{Darboux}
and is based on Darboux' generalization of the work of Dini
(see~\cite[\S601]{Darboux}). 
For recents proofs, see~\cite{MT,dim2,ERA,dedicata}.

\begin{Th} \label{integral}
Let $g$, $\bar g$ be metrics on $M^2$. 
Then they are projectively equivalent 
if and only if the function~$I:TM^2\to \bbR$ defined by
\begin{equation}\label{Integral}
I(\xi ):=\bar g(\xi ,\xi )\left(\frac{\det(g)}{\det(\bar g)}\right)^{2/3}
\end{equation}
is an integral of the geodesic flow of~$g$.
\end{Th}

\begin{Rem} \label{remark 3}
Theorem~\ref{integral} has the corollary 
that each metric~$\bar g$ that is projectively equivalent to~$g$ 
is of the form
\begin{equation*}
\bar g = \left(\frac{\det(g)}{\det(h)}\right)^2 h
\end{equation*}
where~$h\in\mathcal{I}(g)$ satisfies the condition~$\det(h)/\det(g)\ne0$.
\end{Rem}

Another  direct corollary of Theorem~\ref{integral}
(see, for example, \cite{belt,dedicata}, or Section~\ref{proofcorlst} for explanations) 
is the following result due to Knebelman.

\begin{Cor}[\cite{knebelman}] \label{knebel}
Let two metrics $g$ and $\bar g$ on $M^2$ be projectively equivalent. 
Suppose $K=K^i$ is a Killing vector field for $g$. 
Then,
$$\bar K
=\left(\frac{\det(\bar g) }{\det(g)}\right)^{1/3}\bar g^{-1}g(K)
:= \left(\frac{\det(\bar g) }{\det(g)}\right)^{1/3}\bar g^{\alpha j} 
   g_{\alpha i}K^i$$
is a Killing vector field for $\bar g$.
\end{Cor}

\begin{Rem} The mapping $K\mapsto \bar K$,
though linear, is not always a Lie algebra homomorphism.
For example, when~$g$ is flat and~$\bar g$ has nonzero constant
curvature, this mapping clearly cannot be a Lie algebra homomorphism.
\end{Rem}

Koenigs~\cite{konigs} 
proved that superintegrable metrics  always have a Killing vector 
and classified them~\cite[Tableau I]{konigs}.%
\footnote{One must bear in mind, when consulting~\cite{konigs}, 
particularly the Tables, that Koenigs worked over the complex
domain, so that he did not distinguish between the Riemannian and
pseudo-Riemannian cases. It requires a little work (and careful reading
of his notation) to separate out the possible normal forms of
superintegrable Riemannian metrics.}  Though Koenigs  proved this result  locally,   the  global (i.e. on every connected manifold) 
 version of this result  follows  from his list and from the trivial    observation that two Killing vector fields of 
  a metric of nonconstant curvature are proportional. 

The existence of a Killing vector field,  
 combined with Corollary~\ref{knebel}  and with Theorem~\ref{integral} gives

\begin{Lemma} \label{superintegrable1}
If $g$  on $M^2$  is superintegrable, then $\dim\left(\eup(g)\right)=3$.
\end{Lemma}

\begin{proof} Corollary~\ref{knebel}
combined with Theorem~\ref{integral} says that if $K=K^i$ is a
Killing vector field of a metric $g$, 
then for every quadratic integral $F=f_{ij}\xi^i\xi^j$ 
(where $f_{ij}=f_{ji}$), the vector field
\begin{equation} \label{kneb}
Z_F = \frac{\det(f)}{\det(g)}f^{-1}g (K) 
 := \frac{\det(f)}{\det(g)}f^{i\alpha}g_{\alpha j}K^j
\end{equation}
is a projective vector field.

Note that, because, on $2$-by-$2$ matrices, the operation
$f\mapsto \det(f) f^{-1}$ is linear in $f$,
the mapping~$F\mapsto Z_F$ 
is a linear map from~$\mathcal{I}(g)$ to~$\eup(g)$.   Its kernel is at most one-dimensional. Indeed, 
consider   $H(\xi)=h_{ij}\xi^i\xi^j$ with   $h_{ij} =h_{ji}$. Since 
 two symmetric $2$-by-$2$ matrices with the same one-dimensional 
  kernel are proportional, the equality  
$
 Z_F=Z_H=0
$ implies  
  $F= \lambda H$,  or $\lambda F=H$, where $\lambda:M^2\to \mathbb{R}$.  
   If $F$, $H \in \mathcal{I}(g)$, the function 
     $\lambda$ must also be an integral implying it is constant.
  Thus, the   kernel of the mapping~$F\mapsto Z_F$ is at most one-dimensional.  Since $\dim\left(\mathcal{I}(g)\right)=4$  and the kernel of the 
   linear map~$F\mapsto Z_F$ (from $\mathcal{I}(g)$ to $\mathfrak{p}(g)$) 
    is at most one-dimensional,     $\dim\left(\mathfrak{p}(g)\right)\ge 3$.  Since  as we recalled  
     in Section \ref{02} \, metrics of nonconstant curvature have   $\dim\left(\mathfrak{p}\right)\le 3$,   we obtain $\dim\left(\mathfrak{p}(g)\right)= 3$.
Lemma~\ref{superintegrable1} is proved. \end{proof}

\begin{Rem} For use in Section~\ref{proofcorlst},  let us note that the kernel of  the mapping $F\mapsto Z_F$ from the proof of Lemma~\ref{superintegrable1} is the linear hull of the function   
 $F_K(\xi) := \left(g_{ij} K^{i}\xi^j\right)^2\in \mathcal{I}(g)$, which can be checked by direct calculations.   \label{rem.5}
 \end{Rem} 

\begin{Cor} \label{-7} 
If the space of metrics  having a given projective connection 
is more than $3$-dimensional, then each metric $g$ from this 
space has~$\dim\bigl(\eup(g)\bigr) >2$.
\end{Cor}

\noindent{\it Proof.} Because of Theorem~\ref{integral} and Remark~\ref{remark 3}, superintegrable metrics are
those admitting a $4$-parameter family of projectively equivalent
metrics.  By Lemma~\ref{superintegrable1}, they have~$\dim\bigl(\mathfrak{p}\bigr)>2$. Corollary~\ref{-7} is proved. \qed 

Corollary~\ref{-7} is actually what we will need from this section.
Let us note  that the converse statement
is also true.   
\begin{Cor}  \label{corlst}
A metric $g$ 
of nonconstant curvature on  connected  $M^2$  such that   $\dim\bigl(\mathfrak{p}\bigr)=3$  is superintegrable.  
\end{Cor}

We will prove Corollary~\ref{corlst} at the very end of the paper, in Section~\ref{proofcorlst}.

\subsection{ Projective connections admitting two infinitesimal symmetries} \label{connection}

\begin{Lemma} \label{trivial} 
Let the projective connection~\eqref{eq1} admit two infinitesimal symmetries that are linearly independent at the point $p$.
Then there exists a coordinate system~$x,y$ 
in a neighborhood of $p$ such that the vector fields
$X:=(0,1)$ and $Y:= (1,y)$ belong to~$\eup(\eqref{eq1})$. 
In such a coordinate system, the projective
connection~\eqref{eq1} has the form 
\begin{equation} \label{eq2}
y''(x)=A e^{x} + B y'(x)+ C e^{-x} (y'(x))^2+D e^{-2x} (y'(x))^3,
\end{equation} 
where $A,B,C,D$ are constants. 
Moreover, such a coordinate system 
can be chosen so that one of the following conditions holds:
\begin{enumerate} 
\item $D\ne0$, $C=0$, 
\item $D=0$, $C\ne0$, $B=0$, or 
\item $D=C=B=A=0$.
\end{enumerate}
\end{Lemma}

The vector fields $(0,1)$ and $(1,y)$ are always infinitesimal
symmetries of the projective connection~\eqref{eq2}. Therefore,
$\dim\eup(\eqref{eq2})\ge 2$. For certain
values of $A$, $B$, $C$, and~$D$, one can have $\dim\eup(\eqref{eq2})>2$. 
The following lemma describes all such values of $A$, $B$, $C$, and~$D$.

\begin{Lemma} \label{trivial1} 
The following statements hold
\begin{enumerate}

\item The projective connection 
$ y''(x)=A e^{x} + B y'(x)+ De^{-2x} (y'(x))^3 $  with $D\ne 0$
admits a infinitesimal symmetry that is not a
linear combination of the vector fields $(0,1)$ and $(1,y)$ 
if and only if $A=0$ and $B\in \{ 1/2, 2\}$.
Moreover, if $A=0$ and $B=2$, then the projective connection is flat. 
If $A=0$ and $B=1/2$, every infinitesimal symmetry 
is a linear combination of the vector fields $(0,1)$, $(1,y)$, and $(y, y^2/2)$.

\item For $C\ne0$ the algebra of infinitesimal symmetries 
of the projective connection 
$y''(x)=A e^{x} + C e^{-x}(y'(x))^2 $ 
is spanned by the vector fields $(0,1)$ and $(1,y)$.
\end{enumerate}
\end{Lemma}

\noindent {\em Proof of Lemmas~\ref{trivial},\, \ref{trivial1}.}
These lemmas follow from the results of Beltrami, Lie, Cartan,  and Tresse  
that we recalled in Section~\ref{0} and are not new, see 
for example \cite{kowalski,romanovski}. We give
their proofs to make this article self-contained.

We assume that the projective connection is not flat. By
Lemma~\ref{lem3}, in a certain coordinate system $u=(1,y)$ and
$v=(0,1)$ are projective vector fields of this connection.

It is easy to construct the flows of these vector fields.
Indeed, the flow of $v$ is $\Psi_\tau(x,y)=(x,y+\tau)$. 
The flow of $u$ is $\Phi_\tau(x,y)=(x+\tau,e^\tau y)$.
Since the flow of the vector field $u=(0,1)$ preserves the
geodesics, for every solution $y(x)$ and for every $\tau$ the
function $y(x)+\tau$ is also a solution of the equation
\eqref{eq1}. Thus, the coefficients $K^0,K^1,K^2,K^3$ are
independent of $y$.

Similarly, since the flow of the vector field $u=(1,y)$
preserves the geodesics, for every solution $y(x)$ and for every
$\tau$ the function $e^\tau \, y(x+\tau)$ is a solution of the
equation~\eqref{eq1}. Hence, 
$$ e^\tau y''(x+\tau)=K^0 + e^\tau
K^1 y'(x+\tau)+ e^{2\tau} K^2 \ (y'(x+\tau))^2+ e^{3\tau}K^3
(y'(x+\tau))^3.
$$
Thus, the projective connection has the form~\eqref{eq2}.

Now let us show that we can change coordinates so that the
constants $A,B,C,D$ will satisfy conditions 1, 2, or 3. 
Consider the following new coordinate system: $x_\mathrm{new}:=x_\mathrm{old}$,
$y_\mathrm{new}:=y_\mathrm{old}+\alpha e^{x_\mathrm{old}}$. 
In these new coordinates, ~\eqref{eq2} becomes
 \begin{eqnarray*}
y''(x) &=& A_\mathrm{new} e^{x} +B_\mathrm{new}y'(x)+C_\mathrm{new}
e^{-x}(y'(x))^2+D_\mathrm{new} e^{-2x}(y'(x))^3
\\ &=& \left(A-\alpha\,+B\alpha\,+C{\alpha}^{2}+
D {\alpha}^{3} \right)e^{x} +\left( 3\, D
{\alpha}^{2}+2\,C\alpha+B \right) y' \left( x \right)\\
 && + \left( 3\,
D \alpha+C \right){e^{-x}} \left(y' \left( x \right) \right)^{2} +
 D{e^{-2\,x}} \left(y'(x) \right) ^{3} .
\end{eqnarray*}
This new equation has the same form as~\eqref{eq2} and therefore
 the vector fields $(0,1)$, $(1,y)$ are infinitesimal
symmetries of this equation (which is not surprising because the
coordinate change preserves the vector field $(0,1)$ and sends
the vector field $(1,y)$ to a linear combination of $(1,y)$ and
$(0,1)$).

We see that if $D\ne0$, then by the appropriate choice of
$\alpha$ we can make $C_\mathrm{new}=0$ so that the constants
$A_\mathrm{new},B_\mathrm{new}, C_\mathrm{new},D_\mathrm{new}$ will satisfy the first case
from Lemma~\ref{trivial}. Similarly, if $D= 0$ and $C\ne0$, then
by the appropriate choice of $\alpha$ we can make $B_\mathrm{new}=0$ so
that the constants $A_\mathrm{new},B_\mathrm{new}, C_\mathrm{new},D_\mathrm{new}$ will satisfy
the second case of Lemma~\ref{trivial}. If $D=C=0$, then both
Cartan invariants (recalled in Section~\ref{03}) vanish and
therefore, in some coordinate system, $A=B=C=D=0$. 
Lemma~\ref{trivial} is proved. \qed 

Let us now prove Lemma~\ref{trivial1}. Suppose first that the
algebra $\eup(\eqref{eq2})$ is $8$-dimensional. Then, by
the result of Lie, Liouville, Tresse  and Cartan recalled in Section~\ref{03}, 
we have~$L_1=L_2=0$. Substituting the coefficients of the
connection in the formulae for~$L_2$ and~$L_1$, we obtain
\begin{eqnarray*} 
6D(B-2)-2 C^2&=&0\\ 
C+9AD-BC&=&0.
\end{eqnarray*}
Finally, if $B=D=0$ then the second equation implies $C=0$. If
$D\ne0 $ and $C=0$ then the equations imply $A=0$, $B=2$.
Lemma~\ref{trivial1} is proved under the assumption that the
algebra of infinitesimal symmetries is $8$-dimensional.

Now suppose the projective connection~\eqref{eq2} has a
three-dimensional algebra of infinitesimal symmetries.  Then, it is isomorphic to $\eusl(2,\mathbb{R})$, see Section~\ref{02}, and
therefore is generated by three vector fields $X,Y,Z$ satisfying
\[ 
[X,Y]=X\,,\quad [X,Z]=Y\,,\quad [Y,Z]=Z\,.
\]
Without loss of generality, in view of Lemma~\ref{lem3}, 
we can assume that $X=(0,1)$ and $Y=(1,y)$.
Indeed, the vector fields $(0,1)$ and $(1,y)$ satisfy $[X,Y]=X$
and therefore form a Borel subalgebra in the algebra $\eusl(2,\mathbb{R})$ of
infinitesimal symmetries, and all Borel subalgebras of $\eusl(2,\mathbb{R})$
are isomorphic. Then, the vector field $Z$ must satisfy the
conditions
\[ 
[X,Z]=Y\,,\quad [Y,Z]=Z\,.
\]
These conditions are equivalent to the following system of
partial differential equations on the components $Z^1, Z^2$ 
of the vector field $Z$.
$$
\left.
\renewcommand{\arraystretch}{1.5}
\begin{array}{r}
\frac{\p Z^1}{\p y} - 1 = 0\\
\frac{\p Z^2}{\p y} - y =0\\
\frac{\p Z^1}{\p x} + y\frac{\p Z^1}{\p y}
- Z^1 = 0\\
\frac{\p Z^2}{\p x} + y\frac{\p Z^2}{\p y}
- 2Z^2 = 0
\end{array}
\right\}. 
$$ 
Solving the system we obtain $ Z=(Z^1, Z^2)
= \left(y+C_1e^x, \frac{y^2}{2} + C_2e^{2x}\right) $, where $C_1, C_2$
are constants. But if such a vector $Z$ is an infinitesimal
symmetry, then the equations \eqref{ptr} corresponding to the
connection from the second statement of Lemma~\ref{trivial1} imply
the equality 
$$-Ce^{-x}=C_1 C +1= (3A+4CC_2+C_1)e^x=
(4C_2-3C_1A)e^{2x}=0, 
$$ 
which is incompatible. Thus, the second
statement of Lemma~\ref{trivial1} is proved. Similarly, the
equations~\eqref{ptr} corresponding to the connections from the
second statement of Lemma~\ref{trivial1} imply the equality
\[
3C_1De^{-x}= (1-2B-6DC_2)= - (C_1B+C_1+3A)e^x =
(4C_2-3C_1A-2BC_2)e^{2x}=0.
\]
Thus, $ \left\{C_1=0,C_2=-\frac{1}{2D},A=0,B=2\right\} \,,
\text{or} \ \left\{C_1=0,C_2=0,A=0,B=\frac{1}{2}\right\} $. As we
explained above, the first case corresponds to a connection with
8-dimensional space of infinitesimal symmetries.
Lemma~\ref{trivial1} is proved. \qed 

\label{connections}

\subsection{When does a given projective connection come from a metric?}
\label{miracle}

Consider a projective connection
\begin{equation}\label{pc}
y''(x) = K^0(x,y)+K^1(x,y) y'(x)
     +K^2(x,y)\bigl(y'(x)\bigr)^2+K^3(x,y)\bigl(y'(x)\bigr)^3.
\end{equation}
We ask when there exists a metric with this projective connection and,
if any exist, how to find them. Lemma~\ref{1} below, 
due to R. Liouville~\cite[Chapter~III, \S XI]{liouville},
gives a useful tool to answer this question.

Fix a coordinate system $x,y$ (which, for notational ease, 
we shall sometimes refer to as~$(x^1, x^2)$).
Given a metric~$g=E\,dx^2+2F\,dx\,dy+G\,dy^2$, 
construct the symmetric nondegenerate matrix
\begin{equation} \label{a}
a = \left(\begin{array}{cc}a_{11}&a_{12}\\a_{12}&a_{22}\end{array}\right)
:= \det(g)^{-2/3} \, g 
 = \frac1{(EG-F^2)^{2/3}}\left(\begin{array}{cc}E&F\\F&G\end{array}\right).
\end{equation}

\begin{Lemma}[\cite{liouville}]
\label{1}
The projective connection of the metric $g$ is~\eqref{pc} 
if and only if the entries of the matrix~$a$ 
satisfy the \emph{linear} PDE system
\begin{equation} 
\label{lin1} 
\left. 
\renewcommand{\arraystretch}{1.5}
\begin{array}{rcc}
{a_{11}}_x-\tfrac{2}{3}\,K^1\,a_{11} +2\,K^0\,a_{12}&=&0\\
{a_{11}}_y+2\,{a_{12}}_x  
-\tfrac{4}{3}\,K^2\,a_{11}+\tfrac{2}{3}\,K^1\,a_{12}+2\,K^0\,a_{22}&=&0\\
2\,{a_{12}}_y+{a_{22}}_x 
-2\,K^3\,a_{11}-\tfrac{2}{3}\,K^2\,a_{12}+\tfrac{4}{3}\,K^1\,a_{22}&=&0\\
{a_{22}}_y-2\,K^3\,a_{12}+\tfrac{2}{3}\,K^2\,a_{22} &=&0
\end{array} \right\}
\end{equation}
\end{Lemma}

\begin{Rem} 
Before proving the lemma, let us explain how to define the system~\eqref{lin1}
more conceptually. Consider the jet space $J_1^1$ 
(since we are working locally we can think of~$J_1^1$ as $\bbR^3$ with
coordinates $(x,y, y_x)$). Consider the function
$$
F:J^1\to \bbR, \qquad  
F(x,y,y_x):=a_{11} + 2 a_{12} y_x + a_{22} {y_x}^2 ,
$$
(i.e., the function~$F$ is the square of the length of the vector 
$(1,y_x)$ in the quadratic form corresponding to the matrix $a$),
the function
$$
\alpha(x,y, y_x):=\frac{\p y_{xx}}{\p y_x} 
:= K^1 + 2\, K^2 (y_x) + 3 \, K^3 (y_x)^2, 
$$ 
and the `total derivative'~$D_x$ restricted to~\eqref{eq1} 
$$ 
D_x := \frac{\p }{\p x} + y_x \frac{\p}{\p y} 
    + (K^0+K^1y_x+K^2(y_x)^2+K^3(y_x)^3)\frac{\p}{\p y_x}.
$$

Then the system \eqref{lin1} is equivalent to the equation 
$D_x(F) =\frac{2}{3} F \alpha $. More precisely, the equation 
$D_x(F)- \frac{2}{3} F \alpha=0$ is polynomial of third degree
in~$y_x$ and the equations from the system~\eqref{lin1} 
are coefficients of this polynomial.
\end{Rem}

\begin{Rem} 
The proof that we will give below is essentially Liouville's proof,
just in different notation. 

At first glance, the change of variables~\eqref{a} 
that renders 
the obvious equations on~$E$, $F$, and~$G$ 
linear in the unknowns~$a_{ij}$ seems miraculous.%
\footnote{Indeed, Liouville himself seems to have regarded 
this linearity as a remarkably lucky circumstance.} 
However, as Liouville himself noted, 
the occurrence of the matrix $a$ via~\eqref{a}
in Lemma~\ref{1} can be motivated by considering its link 
with the linear equation for quadratic first integrals, 
as we will now explain.

If the projective connection~\eqref{eq1} 
were known to be that of a metric 
associated to a quadratic function $h:TM\to \bbR$, 
then, by Theorem~\ref{integral}, 
a metric~$g$ would have the same geodesics as~$h$
(and hence have~\eqref{eq1} as its projective connection)
if and only if it were to satisfy
\begin{equation} 
\label{123}
\left\{\left({\tfrac{\det(h)}{\det(g)}}\right)^{\tfrac{2}{3}}g,\,h\right\}_h=0,
\end{equation}
where $\{\,,\}_h$ is the Poisson bracket 
transplanted from~$T^*M$ to~$TM$ by the bundle
isomorphism $\flat_h:TM\to T^*M$.
(Note that, although $\det(g)$ depends on a choice of coordinates, 
the ratio $\frac{\det(h)}{\det(g)}$ does not, so~\eqref{123} is
actually a coordinate-independent equation.)

The equation~\eqref{123} is visibly linear in $a=g/\det(g)^{2/3}$, 
which motivates expressing the projective geodesic condition for~$g$ 
in terms of~$a$. 
\end{Rem}

\begin{proof} 
We will first derive a system of PDE on the components of the metric~$g$ 
whose solvability corresponds to the existence of metric with the
projective connection~\eqref{pc} and then show that this system
is equivalent to the system~\eqref{lin1}. 

First of all, it is easy to obtain all symmetric affine
connections whose projective connection is~\eqref{pc}. Indeed, 
by the definition of the projective connection corresponding 
to an affine connection (see Section~\ref{01}), 
the components of the symmetric affine connection 
should satisfy the system of four linear equations 
$$
K^0 = -\Gamma^2_{11},\qquad 
K^1 =  \Gamma^1_{11}-2\Gamma^2_{12},\qquad
K^2 = -\Gamma^2_{22}+2\Gamma^1_{12},\qquad  
K^3 =  \Gamma^1_{22}.
$$ 
Solving this system for the~$\Gamma^i_{jk}$ gives
$$
\renewcommand{\arraystretch}{1.5}
\begin{array}{lll}
\Gamma^2_{11}= -K^0, & 
\Gamma^1_{11}= \phantom{-}K^1+2p_1, &
\Gamma^1_{21}= \Gamma^1_{12}=p_2,\\ 
\Gamma^1_{22}= \phantom{-}K^3,&
\Gamma^2_{22}= -K^2+2p_2, &
\Gamma^2_{21}= \Gamma^2_{12}=p_1,
\end{array}
$$ 
where~$p_1$ and~$p_2$ are functions on~$D^2$.

By definition, $\Gamma_{jk}^i$ is the Levi-Civita connection 
of the unknown metric $g=g_{ij}$, if and only if the covariant
derivative $g_{ij,k}$ in this connection vanishes. The standard
formula 
$$ g_{ij,k}=\frac{\p g_{ij}}{\p x_k}-
\sum_{\alpha=1}^2\left(\Gamma^\alpha_{jk}
g_{i\alpha}+\Gamma^\alpha_{ik} g_{\alpha j}\right) 
$$ 
then gives us the following six first order PDE on five unknown functions, 
which are the components~$E$, $F$, and~$G$ of the metric 
and the functions $p_1$, $p_2$.
\begin{equation} \label{lin3} 
\renewcommand{\arraystretch}{1.5}
\left.
\begin{array}{rcc}
\frac{\p E}{\p x} -2\,K^1\,E +2\,K^0\,F -4\,E\,p_1 &=&0 \\
\frac{\p E}{\p y} -2\,p_2\,E -2\,p_1\,F &=&0 \\
\frac{\p F}{\p y} -K^3\,E +K^2\,F -3\,p_2\,F-p_1\,G &=&0 \\
\frac{\p F}{\p x} -K^1\,F +K^0\,G -3\,p_1\,F-p_2\,E &=&0\\ 
\frac{\p G}{\p x} -2\,p_2\,F -2\,p_1\,G &=& 0\\
\frac{\p G}{\p y} -2\,K^3\,F +2\,K^2\,G -4\,G\,p_2 &=&0
\end{array}\right\}\end{equation}

Thus, the metric $g=
\left(\begin{array}{cc}E&F\\F&G\end{array}\right)$  
has projective connection \eqref{pc} if and only if 
there exist functions $p_1$ $p_2$ such that $E$, $F$, $G$,
$p_1$, and $p_2$ satisfy the system \eqref{lin3}.

Assuming that~$EG-F^2\not=0$, 
solving the second and fifth equation of \eqref{lin3} for $p_1$ and $p_2$, 
and then substituting the result into the remaining equations, 
we obtain a system of four first order PDE 
for the three functions~$E$, $F$, and~$G$, 
whose solvability is equivalent to the existence of a metric 
with the projective connection~\eqref{pc}.

Now making the substitution 
$$
\left(\begin{array}{cc}E&F\\F&G\end{array}\right) 
=\frac{a}{\det(a)^2} 
=
\left(
{\renewcommand{\arraystretch}{1.5}
\begin{array}{cc} 
\dfrac{a_{11}}{(a_{11}a_{22}-a_{12}^2)^2}&  
\dfrac{a_{12}}{(a_{11}a_{22}-a_{12}^2)^2}\\
\dfrac{a_{12}}{(a_{11}a_{22}-a_{12}^2)^2}&
\dfrac{a_{22}\strut}{(a_{11}a_{22}-a_{12}^2)^2}
\end{array}}
\right),
$$
(which is equivalent to~\eqref{a}) into these equations yields,
after some computation, the system~\eqref{lin1}. Lemma~\ref{1} is proved.
\end{proof}

\subsection{Reducing the system~\eqref{lin1} to a system of ODE
and solving it} 
\label{metric}

In this section we assume that our metric~$g$ admits \emph{precisely} 
two linearly independent projective vector fields. Then, by
Lemma~\ref{trivial}, in some  coordinate system near this point the
projective connection of this metric is given as 
\begin{equation}
\label{-6} y''(x)=A e^{x} + B y'(x) +  C e^{-x} (y'(x))^2 +D
e^{-2x} (y'(x))^3.
\end{equation}

By Lemma~\ref{1}, the components $a_{11}$, $a_{12}$, $a_{22}$
of the matrix $a$ constructed for the metric $g$ 
by the formula~\eqref{a} satisfy the system \eqref{lin1}. 
Meanwhile, the coefficients of~\eqref{lin1}, 
considered as a linear first order system for~$a$,
do not depend on $y$ because they are either constants
or constant multiples of the~$K^i$, 
which are independent of $y$ for the projective connection~\eqref{-6}. 

Thus, if $a = (a_{ij})$  is a solution of~\eqref{lin1}, 
then~$\frac{\p a}{\p y}$, $\frac{\p^2a}{\p y^2}$,and~$\frac{\p^3a}{\p y^3}$ 
are also solutions.  However, 
by Corollary~\ref{-7}, our assumption that the space of projective 
vector fields is precisely two-dimensional implies that the space of metrics having \eqref{-6} as projective connection has dimension at most three. 
Thus, there exist constants $\lambda_0,\lambda_1,\lambda_2,\lambda_3$ 
(with at least one of $\lambda_1,\lambda_2,\lambda_3$ being nonzero) 
such that
$$
 \lambda_0\,a
+\lambda_1\,\frac{\p a}{\p y}
+\lambda_2\,\frac{\p^2a}{\p y^2}
+\lambda_3\,\frac{\p^3a}{\p y^3}=0.
$$
Hence, the functions $a_{11}$, $a_{12}$, and $a_{22}$ all satisfy some
linear ordinary differential equation of at most third order
with constant coefficients. 

Depending on the multiplicity  of the roots of the characteristic polynomial 
of this equation, the functions $a_{11}$, $a_{12}$, and $a_{22}$ 
are given by one of the following formulae.
\begin{itemize}
\item[]{\bf Case  1:} 
$a_{11}=c_{01} \, e^{\alpha_1 y} + c_{02}\, e^{\alpha_2 y} 
               +c_{03} \, e^{\alpha_3 y} \,, \quad 
a_{12}=2c_{11} \, e^{\alpha_1y} + 2c_{12}\, e^{\alpha_2 y} 
              + 2c_{13} \, e^{\alpha_3 y} \,,
\quad a_{22}=c_{21} \, e^{\alpha_1 y} + c_{22} \, e^{\alpha_2 y} +
c_{23} \, e^{\alpha_3 y}$,  
where $\alpha_1\in \bbR$, $\alpha_2,\alpha_3\in\mathbb{C}$ 
are mutually different constants and $c_{ij}$ are functions of $x$.

\item[]{\bf Case 2: } $ a_{11}=c_{01} \, e^{\alpha_1 y} + c_{02}
\, e^{\alpha_2 y} + c_{03}y \, e^{\alpha_2 y} \,, \quad
a_{12}=2c_{11} \, e^{\alpha_1 y} + 2c_{12} \, e^{\alpha_2 y} +
2c_{13} y\, e^{\alpha_2 y} \,, \quad a_{22}=c_{21} \, e^{\alpha_1
y} + c_{22} \, e^{\alpha_2 y} + c_{23} y\, e^{\alpha_2 y}$, where
$\alpha_1\in \bbR, \alpha_2\in \mathbb{C}$ are mutually
different constants and $c_{ij}$ are functions of $x$.

\item[]{\bf Case 3:} $ a_{11}=c_{01} \, e^{\alpha y} + c_{02} y \,
e^{\alpha y} + c_{03} y^2 \, e^{\alpha y} \,, \quad a_{12}=2c_{11}
\, e^{\alpha y} + 2c_{12}y  \, e^{\alpha y} + 2c_{13} y^2\,
e^{\alpha y} \,, \quad a_{22}=c_{21} \, e^{\alpha y} + c_{22} y \,
e^{\alpha y} + c_{23} y\, e^{\alpha y}$, where $\alpha\in\bbR$ 
is a constant and $c_{ij}$ are functions of $x$.

\end{itemize}

In Cases 1 and 2, we allow also complex-conjugated $\alpha_i$.

Substituting the above ansatz for $a_{ij}$ in the equations \eqref{lin1}
and using that the functions $y^je^{\alpha_i y}$ are linearly independent 
for different $j$ and $\alpha_i$, we see that the system \eqref{lin1} is equivalent to the following system of ODE on the functions $c_{ij}$.
The system contains 9 ordinary differential equations of the form
\begin{equation}\label{eqM} 
\frac{d}{d x}c = M c, 
\end{equation}
where $c$ is the column with entries 
$(c_{01}, c_{02}, c_{03},c_{11}, c_{12}, c_{13},c_{21}, c_{22}, c_{23})$ 
and $M$ is the following $9$-by-$9$ matrix,\\
{
\tiny 
\centerline{$\left(
\begin{array}{ccccccccc} 
\frac{2}{3}B & 0 & 0 & -Ae^{x} & 0 & 0 & 0 & 0 & 0 \\ 
0 & \frac{2}{3}B & 0 & 0 & -Ae^{x} & 0 & 0 & 0 & 0\\ 
0 & 0 & \frac{2}{3}B & 0 & 0 & -Ae^{x} & 0 & 0 & 0 \\
\frac{4}{3}Ce^{-x}-\alpha _{1} & k_1 & 0 & -\frac{1}{3}B & 0 & 0 &
-2Ae^{x} & 0 & 0 \\ 
0 & \frac{4}{3}Ce^{-x}-\alpha _{2} & k_2& 0 &
-\frac{1}{3}B & 0 & 0 & -2Ae^{x} & 0 \\
0 & 0 &
\frac{4}{3}Ce^{-x}-\alpha _{3} & 0 & 0 & -\frac{1}{3}B & 0 & 0 &
-2Ae^{x} \\
2De^{-2x} & 0 & 0 & \frac{1}{3}Ce^{-x}-\alpha _{1} &
k_1& 0 & -\frac{4}{3}B & 0 & 0 \\
0&2De^{-2x}&0&0&\frac{1}{3}Ce^{-x}-\alpha_{2}&k_2 & 0 & -\frac{4}{3}B & 0 \\
0&0&2De^{-2x}&0&0&\frac{1}{3}Ce^{-x}-\alpha_{3}&0&0& -\frac{4}{3}B
\end{array}
\right)$}
}
and three linear equations on $c_{ij}$, which can be written as
{\begin{equation}\label{3eq}
 {  
\begin{pmatrix} 
\alpha_1 + \frac{2}{3}C e^{-x} & k_1 & 0 \\ 0 & \alpha_2 +
\frac{2}{3}C e^{-x} & k_2 \\ 0 & 0 &  \alpha_3 + \frac{2}{3}Ce^{-x}
 \end{pmatrix} } 
{ \begin{pmatrix}
c_{21} \\ c_{22} \\ c_{23}
\end{pmatrix}}  
= D e^{-2x}\left(\begin{array}{c}c_{11} \\ c_{12}  \\ c_{13}\end{array}\right),
\end{equation}  }
where, the parameters $\alpha_i\in\mathbb{C} , \,  k_i\in \{0, -1, -2\}$  satisfy

\begin{enumerate}
\item 
$\alpha_1\ne\alpha_2\ne\alpha_3\ne\alpha_1$ , $k_1=k_2=0$ for the case 1,
\item 
$\alpha_1\ne\alpha_2=\alpha_3$ , $k_1=0$, $k_2=-1$ for the case 2,
\item 
$\alpha_1=\alpha_2=\alpha_3$ , $k_1=-1$ $k_2=-2$ for the case 3.
\end{enumerate}

Thus, every metric admitting precisely two projective vector fields
comes from the solution of the system above. Note that in
view of Lemmas~\ref{trivial} and \ref{trivial1} we can assume
\begin{itemize}
\item[(a)]   
$C\ne0$ and $D=B=0$,  \label{(b)}
\item[(b)] 
or  $D\ne0$, $C=0$, and if $A=0$, then $B\ne2$ and $B\ne 1/2$. \label{(a)}
\end{itemize}

\begin{Lemma}\label{solution}
Consider the above system of ODE corresponding to one of the
cases 1,2,3. Then, the following holds.

\begin{enumerate}
\item
If the condition (a) holds, the system admits the trivial
solution $c_{ij}\equiv 0$ only. \item If the condition (a) holds
and $A\ne0$, then the system admits the trivial solution
$c_{ij}\equiv 0$ only.

\item If the condition (b) holds and  $A=0$, then
\begin{itemize}

\item[\bf ($\ast$)] 
for $B\ne1$ the general solution of the system corresponds to

\begin{equation}\label{eq.1}
a=\lambda\, \left(
\begin{array}{cc}
e^{\frac{2}{3}Bx} & 0 \\ 0 & \left(
D\frac{e^{2(B-1)x}}{B-1}+H\right) e^{-\frac{4}{3}Bx}
\end{array}
\right)
\end{equation}
where $H\in\bbR$, or \label{Bne1}

\item[\bf ($\ast\ast$)]  \label{B=1} 
for $B=1$ the general solution of the system corresponds to
\begin{equation}\label{eq.2}
a=\lambda \, \left(
\begin{array}{cc}
e^{\frac{2}{3}x} & 0 \\ 0 & \left( 2Dx+H\right) e^{-\frac{4}{3}x}
\end{array}
\right)
\end{equation}
where $H\in\bbR$.

\end{itemize}

\end{enumerate}
\end{Lemma}

\begin{Rem} By \eqref{a} we have
$g=a/(\det(a))^2$. Then, the solution \eqref{eq.1} for $H= 0$
($H\ne0$, respectively)  corresponds to the metric~\eqref{1a} 
(to the metric \eqref{1b}, respectively) from Theorem~\ref{main}, after the
coordinate change $(x_\mathrm{new}, y_\mathrm{new})=(x+c_1,c_2 y)$ 
for the appropriate $c_1,\,c_2 \in \bbR$ and by setting $b:= 2(1-B)$. Recall that  
$B\ne \tfrac{1}{2}, \  1, \  2   $ by assumptions, 
 which implies $b\ne 1, 0, -2$. 
 Similarly, the solution \eqref{eq.2} corresponds to the metric
\eqref{1c} from Theorem~\ref{main}.
\end{Rem}

\noindent {\em Sketch of the proof of Lemma~\ref{solution}.}  Before  giving 
  detailed calculations (below),
  let us  explain  the ideas staying behind, and one more  proof which 
   is actually simpler than one staying below, if we allow    calculations 
   done with  the help of modern computer algebra programs. 

The 'computer algebra' proof is very straightforward: the system \eqref{eqM} 
can be explicitly solved (Maple does it), the general
solution depends on $9$ constants, say $C_1,...,C_9$. Substituting  the solution in the equations \eqref{3eq},
we obtain algebraic relations on the constants $C_1,...C_9$ and $A,B,C,D$. Analyzing these relations,  one  obtains the Lemma.

Now let us explain a trick which allowed us to give a hand-written ( = without the help of computer algebra programs) proof. Differentiate the equations  \eqref{3eq} by $x$ and substitute the values of $\frac{\textrm d}{\textrm{d} x} c_{ij}$ given by \eqref{eqM} inside.  We obtain three new
linear equations on $c_{ij}$.   Repeat the procedure: Differentiate these  new equations   by  $x$ and substitute the values of $\frac{\textrm d}{\textrm{d} x} c_{ij}$ given by \eqref{eqM} inside. We obtain another triple of linear equations on $c_{ij}$. These two triples of equations together with  \eqref{3eq} gives us 9 linear equations on 9 functions $c_{ij}$.  The corresponding $9-$by$-9$ matrix can be explicitly constructed. One immediately sees that under assumptions of the first and the second statements of the Lemma this matrix is nondegenerate almost everywhere   implying $c_{ij}\equiv 0$.
Under assumptions of the third statement of the
 Lemma, the matrix  has rank $7$ implying that the system \eqref{eqM} can be reduced to two ODE, which can be explicitly solved. Their solution gives us the third statement of the Lemma.

\noindent {\em Proof of Lemma~\ref{solution}.} { Assume  $B=D=0$,
$C\ne0$. } Then, the equations \eqref{3eq} imply $c_{21}\equiv
c_{22}\equiv c_{23}\equiv 0$. Substituting this in the last three
equations of \eqref{eqM}, we obtain the system
$$\left(
\begin{array}{ccc}
\frac{1}{3}Ce^{-x}-\alpha _{1} & k_1 & 0  \\ 0 &
\frac{1}{3}Ce^{-x}-\alpha _{2} & k_2  \\  0 & 0 &
\frac{1}{3}Ce^{-x}-\alpha _{3}
\end{array}
\right) \left(
\begin{array}{c}
c_{11} \\ c_{12} \\ c_{13}
\end{array}
\right) = \left(
\begin{array}{c}
0 \\ 0 \\ 0
\end{array}
\right), $$ which evidently implies $c_{11} \equiv  c_{12} \equiv
c_{13}\equiv 0$. Substituting  $c_{11} \equiv  c_{12} \equiv
c_{13}\equiv c_{21}\equiv c_{22}\equiv c_{23}\equiv 0$ in the fourth,
fifth and sixth equations of \eqref{eqM}, we obtain the system

$$\left(
\begin{array}{ccc}\frac{4}{3}Ce^{-x}-\alpha _{1} & k_1 & 0 \\ 
0 & \frac{4}{3}Ce^{-x}-\alpha _{2} &
k_2
\\ 0 & 0 & \frac{4}{3}Ce^{-x}-\alpha _{3}
\end{array}
\right)\left(
\begin{array}{c}
c_{01} \\ c_{02} \\ c_{03}
\end{array}
\right) = \left(
\begin{array}{c}
0 \\ 0 \\ 0
\end{array}
\right),  $$  implying $c_{01} \equiv  c_{02} \equiv c_{03}\equiv
0$. Thus, the system admits only trivial solutions. The first
statement of the lemma is proved.

Now let us prove the second statement of the lemma. We assume
$D\ne0$, $C=0$, $A\ne0$ and show that the system admits only
trivial solutions.

Let us take the thrid, the sixth and the ninth equations of
\eqref{eqM} and the last equation of \eqref{3eq}. We see that
these equations contain only $c_{03}, c_{13}, c_{23}$  and can be
written as follows.

\begin{equation}\label{s2}
 {\small\tiny{\rm \Large \frac{d}{dx}}} \left(
\begin{array}{c}
c_{03} \\ c_{13} \\ c_{23}
\end{array}\right)=\left(
\begin{array}{ccc}
 \frac{2}{3}B & -Ae^{x} & 0 \\-\alpha _{3} &  -\frac{1}{3}B &-2Ae^{x} \\2De^{-2x} & -\alpha _{3}
&-\frac{4}{3}B
 \end{array} \right)
\left(\begin{array}{c}c_{03} \\ c_{13} \\ c_{23} \end{array}\right)
\end{equation}

\begin{equation}\label{s2b}
  \alpha_3 c_{23}=  D e^{-2x}  c_{13}
\end{equation}

Take the equation \eqref{s2b}, differentiate it by $x$ (we obtain
a  linear equation  in  $c_{i3}$ and $\frac{d}{d\, x}c_{i3}$), and
substitute the values of $\frac{d}{d\, x}c_{i3}$ given by
\eqref{s2} inside. We obtain  new linear equations in $c_{i3}$.
Let us play the same game with this new equation, i.e.,
differentiate  by $x$  and substitute the values of $\frac{d}{d\,
x}c_{i3}$  inside. We obtain one more linear equation in $c_{i3}$.
The equation~\eqref{s2b} together with two new obtained equations
can be written as  $M_1c_1=0$, where $c_1$ is the column with
components $c_{03}, c_{13}, c_{23}$ and   $M_1$ is the following
$3$-by-$3$ matrix\\ 
\centerline{\small 
\noindent \begin{tabular}{|c|c|c|}  \hline  $ 0$ & $-{e^{-2\,x}} D
$&${\alpha_3}$
\\ \hline ${\medskip}3\,{\alpha_3}\,{e^{-2\,x}}D$&$-{{\alpha_3}}^{2}+
2\,{e^{-2\,x}}D+1/3\,{e^{-2\,x}}DB $ &$ -4/3\,{
\alpha_3}\,B+2\,{e^{-x}}DA $
\\ \hline  $ {\medskip}{{\alpha_3}}^{3}-{\alpha_3}\,B{e^{-2\,x}}D- $ & $ 5/3\,{{
\alpha_3}}^ {2}B-5\,{e^{-x}}DA{\alpha_3}-4\,{e^{-2\,x}}D- $ &
$  2\,{{\alpha_3}}^{2}A{e^{x}}+{\frac {16}{9}}\,{ 
\alpha_3}\,{B}^{2}
 $ \\ $
8\, {\alpha_3}\,{e^{-2\,x}}D+4\,{e^{-3\,x}}{D}^{2}A $ & $
 4/3\,{e^{-2\,x}}DB-1/9 \,{e^{-2\,x}}D{B}^{2} $ & $
-6\,{e^{-x}}DA-10/3\,{e^{-x}}DAB $ \\ \hline
\end{tabular}}

\vspace{2ex}

  We see that if $A\ne0$,  then  the  determinant of
the matrix is nonzero at almost every point (because the term at
$e^{-6x}$ is equal to $-8\,{D}^{4}{A}^{2}$). Then, $c_{03}\equiv
c_{13}\equiv c_{23}\equiv 0$.

 Let us substitute $c_{03}\equiv c_{13}\equiv c_{23}\equiv 0$ in
the second, fifth and eighth equations of~\eqref{eqM} and in the second
equation of~\eqref{3eq}. We obtain the equation of the form 
$$
 {\small\small \tiny{\rm \Large \frac{d}{dx}}} 
\left(\begin{array}{c}
c_{02} \\ c_{12} \\ c_{22}
\end{array}\right)=\left(
\begin{array}{ccc}
 \frac{2}{3}B & -Ae^{x} & 0 \\-\alpha _{2} &  -\frac{1}{3}B &-2Ae^{x} \\
2De^{-2x} & -\alpha _{2} &-\frac{4}{3}B
 \end{array} \right)
\left(\begin{array}{c}c_{02} \\ c_{12} \\ c_{22}\end{array}\right),
$$ 
$$
  \alpha_2 c_{22}=  D e^{-2x}  c_{12}.
$$ 
We see that these equations are very similar to \eqref{s2} and~\eqref{s2b}
(the only difference is $\alpha_2$ in the place of~$\alpha_3$). 
Arguing as above, we obtain $c_{02}\equiv c_{12}\equiv c_{22}\equiv0$. Substituting$c_{03}\equiv c_{13}\equiv c_{23}
\equiv c_{02}\equiv c_{12}\equiv c_{22}\equiv0$  
in the first, fourth and seventh equations of~\eqref{eqM} and in the
first equation of~\eqref{3eq}, we obtain the equations in the form 
$$ 
{\small\small \tiny{\rm \Large \frac{d}{dx}}}
\left(
\begin{array}{c}
c_{01} \\ c_{11} \\ c_{21}
\end{array}\right)=\left(
\begin{array}{ccc}
 \frac{2}{3}B & -Ae^{x} & 0 \\-\alpha_{1} &  -\frac{1}{3}B &-2Ae^{x} \\
2De^{-2x} & -\alpha _{1}
&-\frac{4}{3}B
 \end{array} \right)\left(
\begin{array}{c}
c_{01} \\ c_{11} \\ c_{21}
\end{array}\right),
$$ 
$$
  \alpha_1 c_{21}=  D e^{-2x}  c_{11}.
$$ 
We see that these equations are again very similar to
\eqref{s2}and~\eqref{s2b} (the only difference is $\alpha_1$ in the
place of $\alpha_3$). Arguing as above, we obtain $c_{01}\equiv
c_{11}\equiv c_{21}\equiv 0$. Finally, the system admits only
trivial solutions. The second statement of the lemma is proved.

Now let us prove the last statement of the lemma.  
Our first goal is to prove that if $\alpha_i\ne0$, 
then  the components $c_{0i}, c_{1i}, c_{2i}$ of the solution are
identically zero.  Indeed, otherwise, arguing as in the proof of
the second statement, we come to the system $M_1c_1=0$, where
$c_1$ is the column with components $c_{03}, c_{13}, c_{23}$ and
$M_1$ is the $3$-by-$3$ matrix from the proof of the second
proposition.

If $\alpha_3\ne0$ then the determinant of the matrix is still
nonzero at almost every point (because the term at $e^{-2x}$ is
equal to $5\,{\alpha_{{2}}}^{4}K \left( B-2 \right) $, and $B\ne
2$ by assumption). Then, the components $c_{03}, c_{13}, c_{23}$
vanish. Doing the same for every $i$, we obtain that if
$\alpha_i\ne0$ then $ c_{0i}\equiv c_{1i}\equiv c_{2i}\equiv0$.

Finally, without loss of generality, we can assume that
$\alpha_1=\alpha_2=\alpha_3=0$, so our system becomes:

\begin{equation} { \tiny {\rm \Large \frac{d}{dx}}
\begin{pmatrix} 
c_{01}
\\ c_{02} \\ c_{03} \\ c_{11}
\\ c_{12} \\ c_{13} \\ c_{21} \\ c_{22} \\ c_{23}
\end{pmatrix} 
= 
\begin{pmatrix} \frac{2}{3}B & 0 & 0 & 0 & 0 & 0
& 0 & 0 & 0 \\ 0 & \frac{2}{3}B & 0 & 0 & 0 & 0 & 0 & 0 & 0
\\ 0 & 0 & \frac{2}{3}B & 0 & 0 & 0 & 0 & 0 & 0 \\
0 & -1 & 0 & -\frac{1}{3}B & 0 & 0 & 0 & 0 & 0
\\ 0 & 0 & -2& 0 & -\frac{1}{3}B & 0
& 0 & 0 & 0 \\ 0 & 0 & 0 & 0 & 0 & -\frac{1}{3}B & 0 & 0 & 0
\\ 2De^{-2x} & 0 & 0 & 0 & -1& 0 &
-\frac{4}{3}B & 0 & 0 \\
0 & 2De^{-2x} & 0 & 0 & 0 & -2 & 0 & -\frac{4}{3}%
B & 0 \\
0 & 0 & 2De^{-2x} & 0 & 0 & 0 & 0 & 0 & -\frac{4%
}{3}B \end{pmatrix} 
\begin{pmatrix} 
c_{01} \\ c_{02} \\ c_{03} \\ c_{11}
\\ c_{12} \\ c_{13} \\ c_{21} \\ c_{22} \\ c_{23} 
\end{pmatrix} 
\label{Mbis} }\end{equation}
\begin{equation}\label{3eqbis}
\begin{pmatrix}
0 & -1 & 0 \\ 0 & 0& -2
\\ 0 & 0 & 0
\end{pmatrix}
\begin{pmatrix} 
c_{21} \\ c_{22} \\ c_{23}
\end{pmatrix}
= D e^{-2x} 
\begin{pmatrix} 
c_{11} \\ c_{12} \\ c_{13}
\end{pmatrix} 
,\end{equation}

This system can easily be solved. Indeed, assume first $B\ne1$.
The second and the third equations of~\eqref{Mbis} give
\begin{equation}\label{first} 
{c_{02}} \left( x \right) ={C_2}\,{e^{2/3\,Bx}},{c_{03}} \left( x \right)
 ={
C_1}\,{e^{2/3\,Bx}} . 
\end{equation} 
The last equation of~\eqref{3eqbis} gives $c_{13}=0$. 
 Substituting all these into the
two last equations of~\eqref{Mbis} gives two ordinary differential
equations 
\begin{equation} \label{one} 
\left.\begin{array}{ccc}
 { \tiny{\rm \Large \frac{d}{dx}}}{c_{12}} \left( x \right) 
+4/3\,B{c_{12}}
\left( x \right) & = &2\,D{C_2}\,{e^{(2/3\,B-2)x}}\\
 {\small\small \tiny{\rm \Large \frac{d}{dx}}}{c_{13}} 
\left( x \right) +4/3\,B{c_{13}}
\left( x\right)& = &2\,D{C_1}\,{e^{(2/3\,B-2)x}}\end{array}\right\}, 
\end{equation}
whose solutions are
\begin{eqnarray}\label{c12}  
c_{12} \left( x \right) 
&=&{e^{-4/3\,Bx}}{C_4}+{\frac {{e^{-2\,x+2/3\,Bx}}D{C_2}}{-1+B}}\\ \label{c12bis}   {c_{23}} 
\left( x\right) 
&=&{e^ {-4/3\,Bx}}{C_3}+{\frac {{e^{-2\,x+2/3\,Bx}}D{C_1}}{-1+B}}. 
\end{eqnarray}
Substituting all these into the fourth and fifth equation of~\eqref{Mbis}, 
we obtain \begin{equation} 
\label{C+0} \left.
\begin{array}{cl} {(B-1) { \left( B-2 \right) {
C_4}\,({e^{(-4/3\,B+2)x}}+1)}} &=D{{{C_2}\, 
\left( 2\,B-1\right) {e^{2/3\,x\left( B-1 \right) }}}}\\
 (B-1)\,{{ \left( B-2 \right) {C_3}\,{e^{-2\,x \left(B-1\right) }}}}&
={D{C_1}\, \left( 2\,B-1 \right) }
 \end{array} \right\}
\end{equation}
implying $C_1=C_2=C_3=C_4=0$. 
Therefore, $c_{02}\equiv c_{03}\equiv c_{12}\equiv c_{13}\equiv 0$, 
which, in view of equations~\eqref{3eqbis}, 
implies $c_{22}\equiv c_{23}\equiv 0$.

Finally, $ c_{02}\equiv c_{03}\equiv c_{12}\equiv c_{13}\equiv
c_{22}\equiv c_{23}\equiv 0$, and the equations\eqref{Mbis}
become 
$$
{\small\small \tiny{\rm \Large \frac{d}{dx}}} 
\left(\begin{array}{c} c_{01} \\ c_{03}\end{array}\right)
=\left(\begin{array}{cc}2/3B & 0 \\ 2 D e^{-2x} & -4/3B\end{array} \right)
\left(\begin{array}{c}c_{01} \\ c_{03}\end{array}\right),
$$ 
implying $c_{01}= \textrm{Const}_1 e^{2/3Bx},
c_{03}=\textrm{Const}_1(\frac{D}{B-1}e^{2(B-1)x}+H)e^{-4/3 B \, x}$
which means that the matrix $a$ is as in Lemma~\ref{solution}.

The case $B=1$ is completely similar. 
In this case instead of \eqref{c12} and \eqref{c12bis} we have:
$$
\begin{array}{c}
{c_{12}(x)}= (2D{C_2}x+{C_4})e^{-4/3x}, \, 
{c_{13}(x)}= (2D{C_1}x+{C_3})e^{-4/3x}.
\end{array}
$$
Substituting in the fifth and sixth equation of~\eqref{Mbis} 
we obtain:
\begin{equation}
\left.
\begin{array}{r}
2D{C_2}e^{-4/3x} - 2D{C_2}e^{-4/3x}x - {C_4}e^{-4/3x}
+ 2{C_1}e^{2/3x}=0\\
\frac{4}{3}(2D{C_1}x + {C_3})e^{-4/3x}=0
\end{array}
\right\}
\end{equation}
implying $C_1=C_2=C_3=C_4=0$. Therefore,
 $c_{02}\equiv c_{03}\equiv c_{12}\equiv c_{13}\equiv 0$, 
which, in view of equations~\eqref{3eqbis}, 
implies $c_{22}\equiv c_{23}\equiv 0$. 
Therefore the equations~\eqref{Mbis} become
$$
{\small\small \tiny{\rm \Large \frac{d}{dx}}} 
\left(\begin{array}{c}c_{01} \\ c_{03}\end{array}\right)
=\left(\begin{array}{cc} 2/3 & 0 \\ 2 D e^{-2x} & -4/3\end{array} \right)
\left(\begin{array}{c}c_{01} \\ c_{03}\end{array}\right),
$$ 
implying $c_{01}= \textrm{Const}_1 e^{2/3x},
c_{03}=\textrm{Const}_1(2Dx+H)e^{-4/3x}$,
which means that the matrix $a$ is as in Lemma~\ref{solution}. 
Lemma~\ref{solution} is proved. \qed 

\subsection{ Metrics admitting precisely three vector fields}

In the previous section we obtained a list of all the metrics
with $\dim(\eup)=2$. The goal of this
section is to obtain the list of the metrics with $\dim(\eup)=2$. Because of Lemmas~\ref{trivial}
and~\ref{trivial1}, without loss of generality, 
we can and will assume that our connection is
\begin{equation} \label{1/2}
 y''(x) = \frac{1}{2} y'(x)+ D
e^{-2x} (y'(x))^3, \qquad\text{where~$D\ne0$},
\end{equation}
and that every projective vector field is a linear combination of
the vector fields
\begin{equation} \label{fields}
X:=(0,1), \ Y:=(1, y), \ Z:=(2y, 1+y^2).
\end{equation}

Let us show that a metric having precisely three projective vector fields 
has a Killing vector field. Because of Corollary~\ref{knebel}, 
it is sufficient to find a metric whose projective connection is~\eqref{1/2}
and that admits a Killing vector field. 
A metric satisfying both properties is
\begin{equation} \label{1/3}
e^{3x}dx^2- 2D \ e^{x} dy^2.
\end{equation}
It admits a Killing vector field because its entries are independent of $y$ 
and its projective connection is~\eqref{1/2}. 
Since only the metrics of constant curvature admit
more than one Killing vector field, 
the Killing vector field is unique up to multiplication by a constant.
As we recalled in Section~\ref{02}, 
the algebra of the projective vector fields is
isomorphic to $\eusl(2,\bbR)$. Recall that the algebra $\eusl(2,\bbR)$ 
is the algebra of the $2$-by-$2$ matrices with trace~$0$. 
It is well-known that up to an automorphism of this algebra every element is proportional to one of the following matrixes:
\begin{equation}\label{matrixes}
{\bf X} := \left(\begin{array}{cc}0& -1\\0&0\end{array}\right), \ \
{\bf Y} := \left(\begin{array}{cc}-1/2 & 0\\0& 1/2\end{array}\right)
\ \ \textrm{or} \ \ 
{\bf Z} := \left(\begin{array}{cc}0&-1\\1&0\end{array}\right).
\end{equation}
Comparing the commutation relations of the vector fields
$X,Y,Z$ with the commutative relation of the matrixes 
${\bf X, Y, Z}$ we see that there exists an isomorphism of the algebras
that sends the vector field $X,Y, Z$ to the matrixes ${\bf X, Y, Z}$, respectively.

Thus, without loss of generality we may assume 
that the Killing vector field is $X$, $Y$, or $Z$.
In the next three sections, we will consider each of these cases in turn.

\subsubsection{Assume $X=(0,1)$ is a Killing vector field.}
Then the components of the matrix $a=g/\det(g)^{2/3}$ 
do not depend on the $y$-coordinate, and the equations~\eqref{lin1} read
$$ 
\left. \begin{array}{rcc}
{\frac {\p {a_{11}}}{\p x}}  -1/3\, {a_{11}} &=&0 \\
 2\,\frac{\p a_{12}}{\p x}  +1/3\, {a_{12}}  &=&0 \\
\frac{\p a_{22}}{\p  x}  +2/3\,{a_{22}}-2\,{D}{a_{11}} &=&0 \\
 -2\,{D}{a_{12}}&=&0
\end{array} \right\}
$$
This system can be solved, the solution is
$$
a_{12}=0, \ \
a_{11}={e^{1/3\,  x} D^{-2/3} 2^{-2/3} C_2},  \ \
a_{22}=(C_1 \, e^{  x}+C_2)2^{1/3}D^{1/3}e^{-5/3}.
$$
Thus, the metric is
\begin{equation*}
g=\left(
\begin{array}{cc}
\frac{e^{3x}}{C_2(C_{1}e^{x}+C_{2})^{2}} & 0 \\
0 & \frac{2De^{x}}{C_{2}^2(C_{1}e^{x}+C_{2})}
\end{array}
\right)
\end{equation*}

We see that if $C_1=0$, then, after an appropriate scaling
$y_\mathrm{new}= c\, y $, the metric
coincides with the metric~\eqref{corol2} from Theorem~\ref{main}.
If $C_1\ne0$, then, after an appropriate coordinate change
$(x_\mathrm{new},y_\mathrm{new})= (x+c_1, c_2\,y)$,  
the metric coincides with the metric~\eqref{corol1} from Theorem~\ref{main}.

\subsubsection{ Assume $Y=(1,y)$ is a Killing vector field. }
Without loss of generality we can assume that $D=\pm 1$. 
Indeed, $D\ne0$ by Lemma~\ref{trivial1}, and,
after the scaling $y_\mathrm{new}= \alpha \, y_\mathrm{old}$, 
the projective connection~\eqref{1/2} becomes 
$$
y''=\frac{1}{2} y' + \frac{D}{\alpha^2}(y')^3.
$$
Consider the coordinate change
$
(x_\mathrm{new}, y_\mathrm{new})
=\left(|y_\mathrm{old}|/e^{x_\mathrm{old}},\ln(|y_\mathrm{old}|)\right)
$.
In this new coordinate system, the vector field $Y$ is $(0,1)$ 
and the projective connection is
$$
y''{}=-\frac{3}{2x}y' + \left(\frac{x}{2}-D x^3\right)(y')^3.
$$
Then the components of the matrix $a=g/\det(g)^{2/3}$ do not depend on the
$y$-coordinate, and the equations~\eqref{lin1} have the form
\begin{equation}  
\left. \begin{array}{rcc}
{\frac{\p a_{11}}{\p  x}} + \frac{{a_{11}}}{x} &=&0 \\
 2\,\frac{\p  a_{12}}{\p  x} - \frac{{a_{12}}}{x} &=&0 \\
{\frac {\p  a_{22}}{\p  x}} -2\frac{{a_{22}}}{x}-(x-2Dx^3){a_{11}} &=&0 \\
 -(x-2Dx^3) {a_{12}} &=&0
\end{array} \right\}
\end{equation}
This system can be solved, the solution $a$ 
and the metric $g=a/\det(a)^2$ are
$$
a=\left(
\begin{array}{cc}
\frac{C_{2}}{x} & 0 \\
0 & {x}(C_{1}\,
x-C_{2}(1+2\,D x^{2}))
\end{array}
\right) $$ \begin{equation} \label{34}  g=\left(
\begin{array}{cc}
\frac{1}{C_2(C_{1}\, x-C_{2}(1+2\,D x^{2}))^{2}x} & 0 \\
0 & \frac{x}{C_2^2(C_{1}\, x-C_{2}(1+2\,D x^{2}))}
\end{array}
\right)
\end{equation} 

We see that 
the metric \eqref{34}  is the metric~\eqref{case2} with $\veps_1\veps_2=-1$ 
from Theorem~\ref{main}, possible after the coordinate change $x_{\mathrm{new}} = -x$.   

\subsubsection{Assume $Z=(2 y, 1+y^2)$ is a Killing vector field.}

As in the previous section, without loss of generality, 
we can assume $D=\pm 1$.

Consider the coordinate change
$$
x_\mathrm{new}=2\frac{y_\mathrm{old}^2+1}{e^{x_\mathrm{old}}}\,, 
\qquad
y_\mathrm{new}=2\arctan(y_\mathrm{old})
$$

In this new coordinate system, the vector field $Z$ is $(0,1)$ 
and the projective connection is
$$
y_{xx}=-\frac{3}{2x}y_x - \left(\frac{x}{2}+D x^3\right)y_x^3.
$$

Then the components of the matrix $a=g/\det(g)^{2/3}$ 
do not depend on the $y$-coordinate, and the equations~\eqref{lin1} are
\begin{equation}  
\left. \begin{array}{rcc}
{\frac{\p{a_{11}}}{\p x}} + \frac{a_{11}}{x} &=&0 \\
 2\,\frac{\p a_{12}}{\p x}  - \frac{a_{12}}{x} &=&0 \\
{\frac{\p{a_{22}}}{\p x}}  -2\frac{a_{22}}{x} -(x+2Dx^3){a_{11}} &=&0 \\
-(x+2Dx^3) {a_{12}} &=&0
\end{array} \right\}
\end{equation}

This system can be solved. 
The solution $a$ and the metric $g=a/\det(a)^2$ are
$$
a=\left(
\begin{array}{cc}
\frac{C_{2}}{x} & 0 \\
0 & {x}(C_{1}\,
x+C_{2}(1-2\,D x^{2}))
\end{array}
\right) \, , 
$$ \begin{equation} \label{35}  
g=\left(
\begin{array}{cc}
\frac{1}{C_2(C_{1}\, x+C_{2}(1-2\,D x^{2}))^{2}x} & 0 \\
0 & \frac{x}{C_2^2(C_{1}\, x+C_{2}(1-2\,´D x^2))}
\end{array}
\right). 
\end{equation} 

We see that,  the metric \eqref{35}  is
the metric~\eqref{case2} with $\veps_1\veps_2=1$ from Theorem~\ref{main}, possible after the coordinate change $x_{\mathrm{new}} = -x$.

\label{2.5}

\subsection{Why are the metrics from Theorem~\ref{main} mutually different?}
\label{different} 

Every metric from Theorem~\ref{main} has its own
individual properties invariant with respect to local isometries
that distinguish it within the metrics from Theorem~\ref{main}.

The individual properties we will use are 
\begin{itemize} 
\item
the structure of the algebra of the projective vector fields and
the place of the Killing vector field in this algebra. 
\item The following three functions: 
\begin{itemize} 
\item Scalar curvature $R:=\sum_{i,j,k} R^i_{ijk} g^{jk}$ 
\item The  square of the length of the derivative of the scalar curvature 
    $I:=\sum_{i,j} g^{ij}
    \frac{\p R}{\p x_i}\frac{\p R}{\p x_j}$
\item The laplacian of the scalar curvature \\
$\Delta_g
     R:=\frac{1}{\sqrt{\det(g)}} \sum_{i,j} \frac{\p }{\p x_i}
     \left( g^{ij} \sqrt{\det(g)}\frac{\p R}{\p x_j}\right)$
\end{itemize}
\end{itemize}

It is easy to distinguish the metrics~\eqref{1a},~\eqref{1b}, and~\eqref{1c} 
from the metrics~\eqref{corol2},~\eqref{corol1}, and~\eqref{case2}.
Indeed, the metrics~\eqref{1a}, \eqref{1b}, and \eqref{1c} 
have a two-dimensional space of projective vector fields 
while the metrics~\eqref{corol2},~\eqref{corol1}, and~\eqref{case2} 
each have a three-dimensional space of projective vector fields.

Let us distinguish the metrics \eqref{corol2}, \eqref{corol1}, 
and \eqref{case2}.  By construction, the Killing vector field 
for the metrics \eqref{corol2}, \eqref{corol1} 
(for the metric \eqref{case2}, respectively), 
is the element of the Lie algebra of projective vector fields 
that is isomorphic to $\eusl(2,\bbR)$, 
is conjugate to {\bf X} from \eqref{matrixes}
({\bf Y} or {\bf Z}, respectively). 
Thus, the metrics \eqref{corol2}, \eqref{corol1} 
cannot be isometric to \eqref{case2}.

Let us distinguish the metrics \eqref{corol2} and \eqref{corol1}. 
By direct calculation, one can see that $\frac{I}{9R^3}$ is equal to $1$
for the metric \eqref{corol2} 
and to ${\frac {8\,\left(e^x+{\veps_2}\right)^4}{\veps_2\, 
\left(3\,e^x+2\,\veps_2\right)^3}}$ for the metric \eqref{corol1}. 
Thus, the metrics \eqref{corol2} and \eqref{corol1} are different.

It is easy to see that two metrics \eqref{corol2} corresponding to different
values of the parameters $\varepsilon_i$ are not mutually isometric.
Indeed, the $\varepsilon_i$ are determined by the signature of the
metric and by the sign of the square of the length of the Killing
vector field.

Now let us show that the metrics \eqref{case2} corresponding to the
different values of the parameters $a$, $c$, and~$\varepsilon_i$ are different.

By direct calculation, we see that for the metrics \eqref{case2} 
the functions of $R$, $I$, $\Delta_gR$ are as follows
\begin{eqnarray*}
R&=& (3c\,
{x}^{2}+4\,{x}^{3}+6{\veps_2}\,x+1/2\,{\veps_2}\,c)/a
\\ I&=& \left( cx+2\,x^{2}+{\veps_2} \right)^{4}x/a^3
\\ \Delta_gR&=& \left( 2\, \veps_2+5c\,x+16\,{x}^{2}
\right)
 \left( cx+2\, x^{2}+\veps_2 \right)^{2}/a^2.\end{eqnarray*} 
We see that the mapping 
$$ 
(R, I, \Delta_gR):\mathbb{R} \setminus\{x\in \mathbb{R} \, : \ \ (cx+2x^2+\varepsilon_2) x=0\} \to \mathbb{R}^3$$ 
is an analytic   curve in $\mathbb{R}^3$, and  can be completed
 at $\{x\in \mathbb{R} \, : \ \ (cx+2x^2+\varepsilon_2) x=0\}$.  
If metrics (2c) corresponding to different values of  the parameters $a, \, c,  \, \veps_2$    are isometric, the images  of the corresponding curves coincide (as subsets of $\mathbb{R}^3$),
 which is not the case. 
 
Indeed, the point $x=0$ is determined by the the condition $I=0$,
$\Delta_gR\ne 0$. At this point $R={\veps_2}{}\frac{c}{2a}$ and
$\Delta_gR=2\veps_{2}^3/a^2$. Since $a> 0$ and $\veps_2=\pm 1$, 
  the curves corresponding to
different values of the parameters  are different,  and therefore the
metrics \eqref{case2} corresponding to different values of the
parameters $a$, $c$ and~$\veps_2$ are different as well.

The remaining parameter  $\varepsilon_1$ determines the role of the Killing vector
field in the Lie algebra of projective vector fields: 
if $\veps_1\veps_2=-1$ ($\veps_1\veps_2=1$, respectively), then the
Killing vector field corresponds to the matrix {\bf Y} ({\bf Z},
respectively) from~\eqref{matrixes}.

Now let us distinguish the metrics \eqref{1a}, \eqref{1b}, \eqref{1c}, 
and \eqref{corol1} corresponding to different values of the parameters.  In order to do this, let us observe that any
isometry between any two of these metrics must send $x$ to $x+x_0$. 
Indeed, the vector $(1,0)$ can be canonically given in isometry-invariant
terms as follows: For every Killing vector field $K$, 
consider the projective vector field $v$ such that
\begin{equation} \label{commu}
[K,v]=K.
\end{equation} 
It is easy to see two such vector fields $u,v$ satisfy $u-v =\lambda\,K$.
Indeed, every Killing vector field for any of these metrics has the
form $\alpha(0,1)$, and every projective vector field $v$ with the
property~\ref{commu} has the form $\beta(0,1) + \gamma(1,y)$.
(In the cases \eqref{1a}, \eqref{1b}, \eqref{1c} 
this is because the space of projective vector fields is two-dimensional. 
In the case \eqref{corol1}, this is because in $\eusl(2,\bbR)$ 
the relations $[K,v]=K$ and $[K,u]=K$ imply that
$K,u, v$ are linearly dependent.) Then, the relation $[K,v]=K$
implies $\gamma=1$, i.e., such vector field $v$ is uniquely
defined up to addition of $\lambda\,(0,1)$

Thus, the projective vector field satisfying the condition~\eqref{commu} 
must have $\gamma=1$, and its projection to the normal distribution 
to the Killing vector field is $(1,0)$. Thus, an isometry 
between any two of the metrics \eqref{1a}, \eqref{1b}, \eqref{1c} 
must send $x$ to $x+x_0$.

By direct calculation, we obtain that the function $R$ for the
metrics \eqref{1a}, \eqref{1b}, \eqref{1c}, and \eqref{corol1}
takes the forms
\begin{eqnarray*}
R_{\eqref{1a}} &=& \veps_1 b \, e^{-(b+2)x} \\ 
R_{\eqref{1b}} &=& \frac{\varepsilon_2 \, b }{2\, a} \left((b+2)e^{-2x} 
           + 2 \veps_2 \,  e^{-(b+2)x}\right)\\
R_{\eqref{1c}} &=& -\frac{1}{2a} (2x+1)e^{-2x}\\
R_{\eqref{corol1}} &=& {\frac{{\veps_2}\,
\left(3\,{e^{x}}+2\,{\veps_2} \right) }{2a{}\,e^{3x}}} .
\end{eqnarray*}
Clearly, the change $x_\mathrm{new}=x+c$ 
cannot translate any of these functions to the same functions corresponding to different values of the parameters. The remaining parameter ($\varepsilon_2$ for (1a), $\varepsilon_1$ for (1b)
 and for (2b), $\varepsilon $  for (1c))
  is determined uniquely by the  sign of the square of the length of the Killing vector. 

Thus, all the metrics from Theorem~\ref{main} are mutually different. 
Theorem~\ref{main} is proved.

\subsection{Proof of Corollary~\ref{corlst}}  

\label{proofcorlst} 
 We will first prove that   if the metrics $g$ and $\bar g$ on $M^2$  are projectively equivalent,  then the spaces $\mathcal{I}( g)$ and $\mathcal{I}(\bar g)$ are   isomorphic\footnote{This statement reflects the fact that the Killing equations are projectively  invariant, see \cite{Benenti3, eastwood1}.}. The canonical 
 isomorphism  is given  by $h\mapsto \left(\frac{\det \bar g}{\det g}\right)^{2/3}\, h$.

Indeed,  the re-parametrization map $$ 
\eta:TM\setminus{M} \to TM \setminus{M}, \  \    \eta(\xi)= \frac{|\xi|_{\bar g}}{|\xi|_{g}}\xi , $$  where 
$TM\setminus{M}$ denotes the tangent bundle without its zero section, takes the orbits of the  geodesic flow of the metric $\bar g$ to the orbits of the geodesic flow of the metric $g$. 
Then, for every $h\in \mathcal{I}(g)$,  the function $$\eta^*h:TM\to \mathbb{R},  \  \  \eta^*h(\xi) =   \left(\frac{|\xi|_{\bar g}}{|\xi|_{g}}\right)^2 h(\xi)=
 \frac{\bar g(\xi, \xi)}{g(\xi,\xi)}\, h(\xi)$$   is constant on the orbits of the geodesic flow of $\bar g$, i.e., is  an    integral of the geodesic flow of $\bar g$.  Using that  the functions $ \bar E(\xi):= \bar g(\xi, \xi)$ and $\bar I(\xi ):= g(\xi ,\xi )\left(\frac{\det(\bar g)}{\det( g)}\right)^{2/3}$ 
 are integrals of the geodesic flow of $\bar g$, we obtain that 
 the    function $\left(\frac{\det (\bar g)}{\det( g)}\right)^{2/3}\, h= \frac{\bar I}{\bar E} \, \eta^*h$ is an integral of the geodesic flow of $\bar g$ as well, i.e., the linear mapping  $h\mapsto  \left(\frac{\det (\bar g)}{\det (g)}\right)^{2/3}\, h$   
 maps  $\mathcal{I}(g)$  to     $\mathcal{I}(\bar g)$. Since we obtain the inverse mapping by interchanging $g$ and $\bar g$,    the mapping $h\mapsto \left(\frac{\det \bar g}{\det g}\right)^{2/3}\, h$ is an isomorphism.

A point $p\in M^2$ will be called {\em regular}, if two vectors from 
$\mathfrak{p}(g)$  are linearly independent at $p$. As we explained in Sections~\ref{02}, \ref{phenomena1},  the set of regular points  is open and  everywhere dense. 

 Let us first  prove Corollary~\ref{corlst} in a small  neighborhood of 
  a regular point.     From  Lemma~\ref{trivial1}  it follows, that if $\mathfrak{p}(g)=3$,      then the  projective connection  has the form  \eqref{1/2} in a certain local  coordinate system near every   regular point.    Since the projective connection of the metric  \eqref{1/3}    is  \eqref{1/2},  it is sufficient to show that the metric \eqref{1/3}  is superintegrable, which is indeed the case since the   functions 
\begin{eqnarray*}
H & = & \tfrac{1}{2}e^{3x} dx^2 - D \, e^{x}dy^2 \\
F_1&=& e^{2x} dy^2  \\
F_2& =& yH+ e^{3x}dxdy \\
F_3&=& y F_2+ 2e^{3x}(y dxdy + 4D^2dy^2)
\end{eqnarray*} 
 are linearly independent integrals of the geodesic flow of the metric \eqref{1/3}.   Thus, the restriction of every metric of nonconstant curvature 
 with $\dim\left(\mathfrak{p}(g)\right)=3$ to a neighborhood of a regular  point is superintegrable.

 Let us now  prove Corollary~\ref{corlst} at every point.  
  We will use the following observation from  \cite{CMH,Topalov}:  \\ {\it If $Z \in  \mathfrak{p}(g)$, then  the function  $I_Z:TM^2\to \bbR$
defined by
$$
I_Z(\xi)=-({\cal L}_Zg)(\xi,\xi)+\frac{2}{3} 
\mathbb{\rm
trace}(g^{-1}{\cal L}_Zg)\, g(\xi,\xi), 
$$  
lies in $\mathcal{I}(g)$.}  

This  observation implies that if $Z \in \mathfrak{p}(g)$ 
is a  Killing vector field on a certain open set, then it is  a Killing vector field everywhere. Indeed, if $I_Z\in \mathcal{I}(g)$ vanishes on the open subset, it vanishes everywhere, since it is constant on the orbits on the geodesic flow.

We have shown that $\dim(\mathcal{I}(g_{|U}))= 4$   
 for  a  certain neighborhood  $U$ of a regular point.
 By the result of Koenigs we recalled in Section~\ref{superintegrable},  in this neighborhood there exists a Killing vector field. 
As we explained above, this implies  the existence of   a Killing vector field $K\in \mathfrak{p}(g)$ defined on the whole $M^2$.

      Consider the linear mapping $F\mapsto Z_F$ from the proof of Lemma~\ref{superintegrable1} and the integral  
$F_K$ from Remark~\ref{rem.5}. Since $\spn(F_K)$  is the kernel of this mapping,   the mapping induces  an isomorphism between $\mathfrak{p}(g)$ and the quotient space  $\mathcal{I}(g_{|U})_{/\spn(F_K)}$. If 
  two  such  neighborhoods  $U_1$ and $U_2$ intersect,
 then the constructed  isomorphisms coincide on the intersection 
   and,  therefore,  induce   an isomorphism between   $\mathfrak{p}(g)$ and $\mathcal{I}(g_{|U_1\cup U_2})_{/\spn(F_K)}$.  Thus, 
   for every connected component $U_\textrm{Reg}$ of the set of the regular points, 
    $\mathfrak{p}(g)$ and $ \mathcal{I}(g_{|U_\textrm{Reg}})_{/\spn(F_K)} $ are isomorphic.  
   
   The set of singular (= not regular) points is obviously 
   invariant w.r.t. Killing vector field. Let us show that the Killing vector field vanishes at singular points.  
   
   Indeed, by Theorem~\ref{main}  
   the universal cover of  $U_{Reg}$ is isometric to a connected domain of $\mathbb{R}^2$ with the metric (2a), (2b), or (2c). We denote the isometry by $\phi$. If $p\in M^n$ approaches  (the lift of)  a singular point, the $x-$coordiate of the  point $\phi(p)$ tends to $\pm \infty$, or to a value  such that 
    the metric  is not defined, i.e., to  the roots of the 
    equation  $e^{x}+\varepsilon_2=0$ for the metric (2b) and of the equation $(cx + x^2 + \varepsilon_2)x=0$ for the metric (2c).
     It is easy to check that 
    in all these cases  the scalar curvature $R$ converges to infinity, or the length of every Killing vector field converges to infinity, or the length of every Killing vector field converges to zero. (The formulae for the 
    scalar curvature are in the previous  section).   Thus, the Killing vector field vanishes  at singular points.

     Then, the set of singular points  consist of a collection of isolated points, so that  
  the set of the regular points has  only     one   connected component $U_\textrm{Reg}$. As we explained above,      $\mathfrak{p}(g)$ and $ \mathcal{I}(g_{|U_\textrm{Reg}})_{/\spn(F_K)} $ are isomorphic implying that $g_{|U_\textrm{Reg}}$ is superintegrable. 
  Since the integrals are preserved by the geodesic flow,  
  the metric $g$ is superintegrable on the whole manifold. Indeed,  for every 
  $I\in \mathcal{I}(g_{|U_\textrm{Reg}})$,  for
   every singular point $p$ and for every geodesic  $\gamma$ such that $\gamma(0)=p$ and   $\dot \gamma(0)\ne 0$,   we put $I(\dot\gamma(0))= I(\dot\gamma(\varepsilon))$, where $1>>|\varepsilon|\ne 0$. The obtained function is evidently smooth, quadratic 
   in velocities and constant on the orbits of the geodesic flow. 
 Corollary~\ref{corlst} is proved.\qed

\end{document}